\documentclass[a4paper, 12pt, reqno]{article}
\usepackage[cp866]{inputenc}
\usepackage[english]{babel}
\usepackage{amssymb, amsmath, amsthm, enumerate}
\usepackage{verbatim}
\usepackage{indentfirst}
\usepackage{graphicx}
\usepackage{amsfonts,amssymb}

 \newtheorem{thm}{Theorem}[section]
 \newtheorem{cor}[thm]{Corollary}
 \newtheorem{lem}[thm]{\bf Lemma}
 \newtheorem{prop}[thm]{\bf Proposition}
 \theoremstyle{definition}
 \newtheorem{defn}[thm]{Definition}
 \theoremstyle{remark}
 \newtheorem{rem}[thm]{\bf Remark}
  \newtheorem{assump}[thm]{\bf Assumption}
 \newtheorem{notat}[thm]{\bf Notation}
  \newtheorem{property}[thm]{\bf Property}
    \newtheorem*{ex}{\bf Example}
 \numberwithin{equation}{section}

\begin{document}

\author
{{\bf  M.~Karmanova, S.~Vodopyanov}\thanks{The research
was partially supported by RFBR (Grants 10-01-00662; 11--01--00819),
and the State
Maintenance Program for the Leading Scientific Schools of the
Russian Federation (Grant No. NSh--6613.2010.1).}}
\title{\sc A Coarea Formula for Smooth Contact Mappings of Carnot--Carath\'{e}odory Spaces}

\maketitle

\begin{abstract}
We prove the coarea formula for sufficiently smooth contact mappings
of Carnot manifolds. In particular, we investigate level surfaces of
these mappings, and compare Riemannian and sub-Riemannian measures
on them. Our main tool  is the sharp asymptotic behavior of the
Riemannian measure of the intersection of a tangent plane to a level
surface and a sub-Riemannian ball. This calculation in particular
implies that the sub-Riemannian measure of the set of characteristic
points (i.\,e., the points at which the sub-Riemannian differential
is degenerate) equals zero on almost every level set.
\end{abstract}

\newpage

\tableofcontents

\newpage

\section{Introduction}

This article is devoted to a sub-Riemannian analog
\begin{equation}\label{srcoarea}
\int\limits_{\mathbb M}{\cal
J}^{SR}_{\widetilde{N}}(\varphi,x)\,d{\cal
H}^{\nu}(x)=\int\limits_{\widetilde{\mathbb M}}\,d{\cal
H}^{\tilde{\nu}}(t)\int\limits_{\varphi^{-1}(t)}\,d{\cal H}^{\nu-\tilde{\nu}}(u)
\end{equation}
of the well-known coarea formula {\eqref{coarea_e}}. Here
$\varphi:\mathbb M\to\widetilde{\mathbb M}$ is a smooth contact
mapping of Carnot--Carath\'{e}odory spaces, and ${\cal
J}^{SR}_{\widetilde{N}}(\varphi,x)$ is a sub-Riemannian coarea
factor determined by the values of a sub-Riemannian
differential, or {\it$hc$-differential}. In some sense, we may
regard {\eqref{srcoarea}} as a generalization of the results of
 {\cite{P, H, M00,M05}} and some other articles.

It is well known that the coarea formula
\begin{equation}\label{coarea_e}
\int\limits_{U}{\mathcal J}_k(\varphi,x)\,dx=\int\limits_{\mathbb
R^k}\,dz\int\limits_{\varphi^{-1}(z)}\,d{\mathcal H}^{n-k}(u),
\end{equation}
where ${\mathcal
J}_k(\varphi,x)=\sqrt{\det(D\varphi(x)D\varphi^*(x))}$, has many
applications in analysis on Euclidean spaces
{\cite{kr,fd1,F,eg,gq,fx,ot}}. Here we assume that $\varphi\in
C^1(U,\mathbb R^k)$ with $U\subset\mathbb R^n$
 for $n\geq k$. In
particular, {\eqref{coarea_e}}  applies in the theory of exterior forms and currents, and in  problems about minimal surfaces (see, for
example, {\cite{FF}}). Also, Stokes' formula easily follows from the
coarea formula (see, for instance, {\cite{V1}}). The development of
analysis on more general structures raises a natural question of
extending the coarea formula to  objects with a more general
geometry than in Euclidean spaces, especially to metric spaces and
sub-Riemannian manifolds. In 1999, L.~Ambrosio and B.~Kirchheim
{\cite{AK}} proved an analog of the coarea formula for Lipschitz
mappings defined on an ${\cal H}^n$-rectifiable metric space with
values in $\mathbb R^k$ for $n\geq k$. In 2004, this formula was
established for Lipschitz mappings defined on an ${\cal
H}^n$-rectifiable metric space with values in an ${\cal
H}^k$-rectifiable metric space for $n\geq k$ {\cite{Km1, Km3}}.
Moreover,  conditions  were found on the image and preimage of a
Lipschitz mapping defined on an $\mathcal H^n$-rectifiable metric
space with values in an {\it arbitrary} metric space which are
necessary and sufficient for the validity of the coarea formula.
Independently of that result, the level sets of those mappings were
investigated, and the metric analog of implicit function theorem was
proved {\cite{Km2, Km3, Km4}}.

All results mentioned above were obtained on rectifiable metric spaces.
Note that their metric structure is similar to that of Riemannian
manifolds. However,  there exist {\it non-rectifiable} metric spaces
whose geometry is not comparable to the Riemannian one. {\it Carnot
manifolds} are of special interest (for results on their
non-rectifiability, see, for instance, {\cite{AK}}). Sub-Riemannian
geometry naturally arises in the theory of subelliptic equations,
contact geometry, optimal control theory, non-holonomic mechanics,
neurobiology, and other areas (see, e.~g.,
{\cite{nsw,G,B,M00,MM,MM1,AM,J,vg,vod1,vod2,CS, HP}}). This theory
has many applications. In addition, it has many well-known unsolved
problems.

One of them is the problem of the sub-Riemannian coarea formula,
which is useful for developing a non-holonomic theory of currents,
exterior forms, extremal surfaces (in sub-Riemannian and
sub-Lorentzian geometries), and so forth.

Heisenberg groups and Carnot groups are well-known particular
cases of Carnot manifolds. In 1982, P.~Pansu proved the coarea
formula for real-valued functions defined on a Heisenberg group
{\cite{P}}. Next, J.~Heinonen {\cite{H}} extended this formula to
smooth functions defined on a Carnot group. Another result
concerning the analog of {\eqref{coarea_e}} belongs to V.~Magnani.
In 2000, he proved a {\it coarea inequality} for mappings of Carnot
groups {\cite{M00}}. The equality was established for a mapping from
a Heisenberg group to the Euclidean space~$\mathbb
R^k$~{\cite{M05}}. The validity of
coarea formula even for a model case of a mapping from a
Carnot group to other Carnot group has remained an open question.

The purpose of this paper is to prove the coarea formula for
sufficiently smooth contact mappings $\varphi:\mathbb M\to\widetilde{\mathbb M}$ of
Carnot manifolds. We emphasize that all results are new even
in the particular case of a mapping of Carnot groups.

As we mentioned above, for the first time a non-holonomic analogue of the coarea formula was discovered by P.~Pansu {\cite{P}}.
One of the basic ideas was to prove the sub-Riemannian coarea formula via the Riemannian one:
\begin{multline}\label{riemsub}
{\eqref{coarea_e}}\Rightarrow\int\limits_{U}{\mathcal
J}^{SR}_{\widetilde{N}}(\varphi,x)\,d\mathcal
H^{\nu}(x)\\=\int\limits_{\widetilde{\mathbb M}}\,d\mathcal
H^{\tilde{\nu}}(z)\int\limits_{\varphi^{-1}(z)}\frac{\mathcal
J^{SR}_{\widetilde{N}}(\varphi,u)}{\mathcal J_{\widetilde{N}}(\varphi,x)}\,d{\mathcal
H}^{N-\widetilde{N}}(u)\overset{?}=\int\limits_{\widetilde{\mathbb M}}\,d\mathcal
H^{\tilde{\nu}}(z)\int\limits_{\varphi^{-1}(z)}\,d{\mathcal
H}^{\nu-\tilde{\nu}}(u)
\end{multline}
Here $N$ and $\widetilde{N}$ are the topological dimensions while $\nu$ and $\tilde{\nu}$ are the Hausdorff dimensions of the preimage and image respectively. It is clear that in sub-Riemannian spaces these dimensions differ.
Other authors subsequently used the same idea.

It follows easily from {\eqref{riemsub}} that the crucial point in this method is to understand the relation between the Riemannian and sub-Riemannian measures on the Carnot manifolds themselves and on the level sets. Moreover, an appropriate definition of the sub-Riemannian coarea factor is required. It is well known that the question of measure relation on the given spaces is trivial, whereas those of the geometry of level sets and the sub-Riemannian coarea factor are quite complicated. The main difficulties are related to the peculiarities of sub-Riemannian metrics. The inequivalence of Riemannian and sub-Riemannian metrics is illustrated for example by the fact that the Riemannian distance from the center of a radius $r$ sub-Riemannian ball to its boundary varies from $r$ to $r^M$, $M>1$, where the constant $M$ depends on the structure of the space. Thus, an immediate question arises: how ``sharply'' does the tangent plane approximate a level set? (This question is important since the Riemannian tangency order $o(r)$ is insufficient here: a level surface can ``jump'' out of a sub-Riemannian ball much earlier than necessary if this order of tangency is higher than $o(r)$.) Also, it is a problem whether or not there exists a metric allowing us to describe the shape and geometry of the intersection of a sub-Riemannian ball and a level surface. Even if we answer these two questions, the most difficult question concerns the relation between the Hausdorff dimensions of the image and the intersection of a ball and a level set.

In this article we solve problems stated above. Firstly, we divide the set of all points at which the classical differential is non-degenerate into a regular set and a characteristic set. Next, we define a sub-Riemannian quasimetric~$d_2$ that enables us to calculate the measure of the intersection of a sub-Riemannian ball and a tangent plane (see below Remark~{\ref{tangplane}} for details). A crucial idea in the construction of $d_2$ is based on the fact that a ball in this quasimetric is asymptotically equal to direct product of Euclidean balls:
$$
B^{N}_{d_2}(x,r)\approx B^{n_1}(x,r)\times
B^{n_2}(x,r^2)\times\ldots\times B^{n_M}(x,r^M), \ M>1,
$$
where $N$ is the topological dimension of a Carnot manifold, and $n_i$ are the (topological) dimensions of the Euclidean balls $B^{n_i}$, $i=1,\ldots, M$. Thus, if a plane intersects a ball of this type then we can easily determine the shape of the intersection (contrary to the case of a ``box'' quasimetric with
$$\operatorname{Box}(x,r)\approx Q^{n_1}(x,r)\times Q^{n_2}(x,r^2)\times\ldots\times Q^{n_M}(x,r^M),
$$
where we have Euclidean cubes $Q^{n_i}$, $i=1,\ldots, M$, instead of balls, since the cubes have sections of different shapes). Studying the sharpness of the approximation to a level surface by its tangent plane, we introduce a ``mixed'' quasimetric possessing both Riemannian and sub-Riemannian properties. We prove that at regular points the tangent plane approximates the level surface sufficiently well, and this fact enables us to calculate the (Riemannian) measure of this intersection, which depends on the Hausdorff dimensions of both the image and preimage. In other words, it is equivalent to $r^{\nu-\tilde{\nu}}$ (see Theorems
{\ref{tang_meas}} and {\ref{Leb_meas}}). These results yield analytic expression of the relation between the Riemannian and sub-Riemannian measures at the regular points on level sets (see Theorem~{\ref{derivativemeas}}).

The characteristic set case is a little more complicated since precisely near a characteristic point a surface can ``jump'' out of a sub-Riemannian ball. For this reason we cannot estimate the measure of the intersection of a ball and a surface via the measure of the intersection of a ball and the tangent plane at a characteristic point.
Note also that in all papers mentioned above the preimage has a group structure, whose properties are essential for proving that the measure of the set of all characteristic points on each level surface equals zero. In the case of a mapping of two Carnot manifolds, both the image and preimage lack a group structure, and the approximation of a manifold by its local Carnot group is insufficient for the extension of the methods developed to our case. That is why we create an ``intrinsic'' method for studying the properties of the characteristic set.

The result on the characteristic set is stated in Theorem~{\ref{charsetth}}.

In Section 4 we prove that the degenerate set of the classical differential does not affect either part of the coarea formula.

Finally, in Section 5, we define the sub-Riemannian coarea factor (via the values of the $hc$-differential) and derive the sub-Riemannian coarea formula~{\eqref{srcoarea}}.

For mappings $\varphi:\mathbb H^1\to\mathbb R$, formula {\eqref{srcoarea}} becomes
$$
\int\limits_{\mathbb M}|\nabla_H\varphi(x)|\cdot\frac{\omega_3^2}{\omega_4\cdot 4}\,d{\cal
H}^{\nu}(x)=\int\limits_{\widetilde{\mathbb M}}\,d{\cal
H}^{\tilde{\nu}}(t)\int\limits_{\varphi^{-1}(t)}\,d{\cal H}^{\nu-\tilde{\nu}}(u)
$$
since in this case we have
$${\cal J}^{SR}_{\widetilde{N}}(\varphi,x)=|\nabla_H\varphi(x)|\cdot\frac{\omega_3^2}{\omega_4\cdot 4}.
$$
Here $\nabla_H\varphi(x)$ is the ``horizontal part'' of the gradient of $\varphi$ at $x$, and the coefficient $\frac{\omega_3^2}{\omega_4\cdot 4}$ appears due to the choice of a constant factor in the definition of the Hausdorff measure (to define $\mathcal H^{\mu}$ we set this factor to $\omega_{\mu}$, and if we consider $2^{\mu}$ instead of $\omega_{\mu}$ then this factor will be equal to 1); cf.~{\cite{P}}. For $\varphi:\mathbb H^n\to\mathbb R^k$, $k\leq 2n$, we deduce that
\begin{multline*}
\int\limits_{\mathbb M}\sqrt{\det(D_H\varphi(x)D_H\varphi(x)^*)}\cdot\frac{\omega_{2n+1}\cdot\omega_{2n+2-k}}{\omega_{2n+2}\cdot\omega_{2n-k}\cdot2}\,d{\cal
H}^{\nu}(x)\\
=\int\limits_{\widetilde{\mathbb M}}\,d{\cal
H}^{\tilde{\nu}}(t)\int\limits_{\varphi^{-1}(t)}\,d{\cal H}^{\nu-\tilde{\nu}}(u)
\end{multline*}
and
$${\cal J}^{SR}_{\widetilde{N}}(\varphi,x)=\sqrt{\det(D_H\varphi(x)D_H\varphi(x)^*)}\cdot\frac{\omega_{2n+1}\cdot\omega_{2n+2-k}}{\omega_{2n+2}\cdot\omega_{2n-k}\cdot2}.
$$
Here $D_H\varphi(x)$ is the ``horizontal part'' of the differential $D\varphi(x)$, and if we replace all occurrences of $\omega_{\mu}$ by $2^{\mu}$ then the coefficient following the Gram determinant of $D_H\varphi(x)$ will be equal to 1; see also {\cite{M05}}. Similarly, for $\mathbb M=\mathbb G$ and $\widetilde{\mathbb M}=\mathbb R$ we obtain
$$
\int\limits_{\mathbb M}|\nabla_H\varphi(x)|\cdot\frac{\omega_N\cdot\omega_{\nu-1}}{\omega_\nu\cdot \prod\limits_{k=2}^M\omega_{n_k}\cdot\omega_{n_1-1}}\,d{\cal
H}^{\nu}(x)=\int\limits_{\widetilde{\mathbb M}}\,d{\cal
H}^{\tilde{\nu}}(t)\int\limits_{\varphi^{-1}(t)}\,d{\cal H}^{\nu-\tilde{\nu}}(u)
$$
and
$${\cal J}^{SR}_{\widetilde{N}}(\varphi,x)=|\nabla_H\varphi(x)|\cdot\frac{\omega_N\cdot\omega_{\nu-1}}{\omega_\nu\cdot \prod\limits_{k=2}^M\omega_{n_k}\cdot\omega_{n_1-1}}.
$$
Once again, it is easy to see that if we take $2^{\mu}$ instead of $\omega_{\mu}$ then the coefficient following the length of the horizontal gradient of $\varphi$ at $x$ will be equal to 1;  compare with {\cite{H}}.

Thus, for sufficiently smooth mappings we obtain all previous results as particular cases.

\section{Preliminaries}

In this section we introduce some necessary definitions and mention important facts used to prove the main result.

\begin{defn}[cf.~{\cite{G,vk_birk,nsw}}]\label{carnotmanifold}
Fix a connected Riemannian
$C^{\infty}$-mani\-fold~$\mathbb M$ of topological dimension~$N$. The
manifold~$\mathbb M$ is called a {\it Carnot--Carath\'{e}o\-dory space} if the
tangent bundle $T\mathbb M$ has a filtration
$$
H\mathbb M=H_1\mathbb M\subsetneq\ldots\subsetneq H_i\mathbb M\subsetneq\ldots\subsetneq H_M\mathbb M=T\mathbb M
$$
by subbundles such that each point $p\in\mathbb M$ has a neighborhood $U\subset\mathbb M$
equipped with a collection of $C^{1,\alpha}$-smooth vector fields
$X_1,\dots,X_N$, $\alpha\in(0,1]$, enjoying the following two properties. For each
$v\in U$,

$(1)$ $H_i\mathbb M(v)=H_i(v)=\operatorname{span}\{X_1(v),\dots,X_{\dim H_i}(v)\}$
is a subspace of $T_v\mathbb M$ of a constant dimension $\dim H_i$,
$i=1,\ldots,M$;

$(2)$ we have
\begin{equation}\label{tcomm}[X_i,X_j](v)=\sum\limits_{k:\,\operatorname{deg}
X_k\leq \operatorname{deg} X_i+\operatorname{deg}
X_j}c_{ijk}(v)X_k(v)
\end{equation}
where the {\it degree} $\deg X_k$ is defined as $\min\{m\mid X_k\in H_m\}$;

Moreover, if the third condition holds then the Carnot--Carath\'eodory space is called the {\it Carnot manifold}:

$(3)$ the quotient mapping $[\,\cdot ,\cdot\, ]_0:H_1\times
H_j/H_{j-1}\mapsto H_{j+1}/H_{j}$ induced by the Lie bracket is an
epimorphism for all $1\leq j<M$.

The subbundle $H\mathbb M$ is called {\it horizontal}.

The number $M$ is called the {\it depth} of the manifold $\mathbb
M$.
\end{defn}

Properties of Carnot-Carath\'eodory spaces and Carnot manifolds
under assumptions of regularity mentioned in  definition
\ref{carnotmanifold} can be found in \cite{vk_birk, k_dan_gafa, k_gafa, vk_locappr, k_gr}

\begin{ex}
A Carnot group is an example of a Carnot manifold.
\end{ex}

\begin{defn}
Consider the initial value problem
$$
\begin{cases}
\dot{\gamma}(t)=\sum\limits_{i=1}^Ny_iX_i(\gamma(t)),\ t\in[0,1],\\
\gamma(0)=x,
\end{cases}
$$
where the vector fields $X_1,\ldots,X_N$ are $C^1$-smooth.
Then, for the point $y=\gamma(1)$ we write $y=\exp\Bigl(\sum\limits_{i=1}^Ny_iX_i\Bigr)(x)$.

The mapping $(y_1,\ldots,y_N)\mapsto\exp\Bigl(\sum\limits_{i=1}^Ny_iX_i\Bigr)(x)$ is called the {\it exponential}.
\end{defn}

\begin{defn}\label{cs1kind}
Consider $u\in\mathbb M$ and $(v_1,\ldots, v_N)\in B_E(0, r)$,
where $B_E(0,r)$ is a Euclidean ball in $\mathbb R^N$. Define a
mapping $\theta_u:B_E(0,r)\to\mathbb M$ as
follows:
$$
\theta_u(v_1,\ldots,
v_N)=\exp\biggl(\sum\limits_{i=1}^Nv_iX_i\biggr)(u).
$$
It is known that $\theta_u$ is a $C^1$-diffeomorphism if $0<r\leq
r_u$ for some $r_u>0$. The collection $\{v_i\}_{i=1}^N$ is called
{\it the normal coordinates} or {\it the coordinates of the
$1^{\text{st}}$ kind  $($with respect to $u\in\mathbb M)$} of the
point $v=\theta_u(v_1,\ldots, v_N)$.
\end{defn}

\begin{thm}[\cite{vk_birk}]
\label{Liealg} Fix $u\in\mathbb M$. The coefficients
$$
\bar{c}_{ijk}=
\begin{cases} c_{ijk}(u)\text{ of \eqref{tcomm} }& \text{if }\operatorname{deg}
X_i+\operatorname{deg}X_j=\operatorname{deg}X_k, \\
0 &\text{otherwise}
\end{cases}
$$
define a graded nilpotent Lie algebra.
\end{thm}

We construct the Lie algebra $\mathfrak{g}^u$ of Theorem
{\ref{Liealg}}  as a graded nilpotent Lie algebra of vector fields
$\{(\widehat{X}_i^u)^{\prime}\}_{i=1}^N$ on $\mathbb R^N$ such that
the exponential mapping $(x_1,\ldots,
x_N)\mapsto\exp\Bigl(\sum\limits_{i=1}^Nx_i(\widehat{X}_i^u)^{\prime}\Bigr)(0)$
is the identity~{\cite{post, blu}}. In view of the results of {\cite{fs}}, the value of $(\widehat{X}_j^u)^{\prime}(0)$ is equal to a standard vector $e_{i_j}\in\mathbb R^N$, where $i_j\neq i_k$ if $j\neq k$, $j=1,\ldots, N$. We associate to each vector field of the resulting collection an index $i$ so that $(\theta_u)_*\langle(\widehat{X}_i^u)^{\prime}\rangle(u)=X_i(u)$.
By the construction, the vector fields
$\{(\widehat{X}_i^u)^{\prime}\}_{i=1}^N$ satisfy
\begin{equation}\label{tcommnilp}
[(\widehat{X}_i^u)^{\prime},(\widehat{X}_j^u)^{\prime}]=\sum\limits_{\operatorname{deg}
X_k=\operatorname{deg} X_i+\operatorname{deg}
X_j}c_{ijk}(u)(\widehat{X}_k^u)^{\prime}
\end{equation}
everywhere on $\mathbb
R^N$.

\begin{notat}
We use the following standard notation: for each $N$-dimen\-sional
multi-index $\mu=(\mu_1,\ldots,\mu_N)$, its {\it homogeneous norm}
equals $|\mu|_h=\sum\limits_{i=1}^N\mu_i\operatorname{deg} X_i$.
\end{notat}

\begin{defn}
\label{groupoperator} Refer as the graded nilpotent Carnot group  ${\mathbb G}_u\mathbb M$
corresponding to the Lie algebra $\mathfrak{g}^u$ to the
{\it nilpotent tangent cone} of  $\mathbb M$ at $u\in\mathbb M$.
We construct ${\mathbb G}_u\mathbb M$ in $\mathbb R^N$ as a groupalgebra
{\cite{post}},
that is, the exponential map is the identity:
$$
\exp\Bigl(\sum\limits_{i=1}^Nx_i(\widehat{X}_i^u)^{\prime}\Bigr)(0)=(x_1,\ldots,x_N).
$$
 By the Baker--Campbell--Hausdorff formula, the group operation
is defined so that the basis vector
fields~$(\widehat{X}_i^u)^{\prime}$ on $\mathbb R^N$,
$i=1,\ldots,N$, are left-invariant~{\cite{post}}: if
$$x=\exp\Bigl(\sum\limits_{i=1}^Nx_i(\widehat{X}_i^u)^{\prime}\Bigr),\ y=\exp\Bigl(\sum\limits_{i=1}^Ny_i(\widehat{X}_i^u)^{\prime}\Bigr)$$
then
$$
x\cdot
y=z=\exp\Bigl(\sum\limits_{i=1}^Nz_i(\widehat{X}_i^u)^{\prime}\Bigr),
$$

where
\begin{align}\label{group}
z_i&=x_i+y_i,\quad \operatorname{deg} X_i=1,\notag
\\
z_i&=x_i+y_i+\sum\limits_{\substack{|e_l+e_j|_h=2,\notag\\
l<j}}{F}^i_{e_l,e_j}(u)(x_ly_j-y_lx_j), \quad\operatorname{deg} X_i=2,
\\
z_i&=x_i+y_i+ \sum\limits_{\substack{|\mu+\beta|_h=k,\\
\mu>0,\,\beta>0}}{F}^i_{\mu,\beta}(u) x^\mu\cdot y^\beta
\\
&=x_i+y_i+\sum\limits_{\substack{|\mu+e_l+\beta+e_j|_h=k,\\
l<j}} {G}^i_{\mu,\beta,l,j}(u)
x^{\mu}y^{\beta}(x_ly_j-y_lx_j),\quad \operatorname{deg} X_i=k.\notag
\end{align}
\end{defn}

Using the exponential mapping $\theta_u$, we can push forward the vector
fields $(\widehat{X}_i^u)^{\prime}$ onto $\mathcal U\subset\mathbb M$ as
$$[(\theta_u)_*\langle (\widehat{X}_i^u)^{\prime}\rangle](\theta_u(x))=D\theta_u(x)\langle
(\widehat{X}_i^u)^{\prime}(x)\rangle$$
and obtain the vector fields
$\widehat{X}_i^u=(\theta_u)_*(\widehat{X}_i^u)^{\prime}$. Recall that
$\widehat{X}_i^u(u)=X_i(u)$.

\begin{defn}\label{locCargr}
Associated to the Lie algebra $\{\widehat{X}_i^u\}_{i=1}^N$ at $u\in\mathbb M$,
is a {\it local homogeneous group} ${\mathcal
G}^u\mathbb M$. Define it so that the mapping $\theta_u$
is a {\it local group isomorphism} between some neighborhoods of the identity elements
of the groups $\mathbb G_u\mathbb M$ and ${\mathcal G}^u\mathbb M$.

The canonical Riemannian structure on ${\mathcal G}^u\mathbb M$ is
determined by the inner product at the identity element of ${\mathcal G}^u\mathbb M$
coinciding with that on $T_u\mathbb M$. The canonical Riemannian
structure on the nilpotent tangent cone  ${\mathbb G}_u\mathbb M$ is
defined so that the local group isomorphism $\theta_u$ is
an isometry.
\end{defn}

\begin{assump}
Henceforth we assume that the neighborhood $\mathcal U$
under consideration is such that $\mathcal U\subset\mathcal G^u\mathbb M$
for all $u\in \mathcal U$.
\end{assump}

\begin{defn}\label{defhoriz}
A curve $\gamma:[0,1]\to\mathbb M$ which is absolutely continuous in the Riemannian sense is called {\it horizontal} if $\dot{\gamma}(t)\in H_{\gamma(t)}\mathbb M$ for almost all $t\in [0,1]$ with respect to the Lebesgue measure on $[0,1]$.

A horizontal curve in $\mathcal G^u\mathbb M$ is defined similarly:
here we require that $\dot{\gamma}(t)\in\operatorname{span}\{\widehat{X}_1^u(\gamma(t)),\ldots,\widehat{X}_{\dim H_1}^u(\gamma(t))\}=\widehat{H}^u_{\gamma(t)}\mathbb M$
for almost all $t\in [0,1]$ with respect to the Lebesgue measure on $[0,1]$.
\end{defn}

\begin{thm}[see {\cite{rash, chow}} for smooth case and {\cite{vk_birk, k_gafa}} for $C^{1,\alpha}$-smooth case] Every two points of $\mathbb M$ can be joined by a horizontal curve.

\end{thm}

\begin{defn}\label{defdcc}
Given $x,y\in\mathbb M$, the Carnot--Carath\'{e}odory distance $d_{cc}(x,y)$ is defined as
$$
d_{cc}(x,y)=\inf\{\ell(\gamma): \gamma:[0,1]\to\mathbb M,\ \dot{\gamma}(t)\in H_{\gamma(t)}\mathbb M\},
$$
where the length $\ell$ of each (horizontal) curve $\gamma$ is calculated with respect to the Riemannian tensor on $\mathbb M$.

For $x,y\in\mathcal G^u\mathbb M$ the Carnot--Carath\'{e}odory distance $d^u_{cc}(x,y)$ is defined as
$$
d^u_{cc}(x,y)=\inf\{\ell^u(\gamma): \gamma:[0,1]\to\mathcal G^u\mathbb M,\ \dot{\gamma}(t)\in \widehat{H}^u_{\gamma(t)}\mathbb M\},
$$
where the length $\ell^u$ of each (horizontal) curve $\gamma$ is calculated with respect to the Riemannian tensor on $\mathcal G^u\mathbb M$.

Denote the ball of radius $r$ in $d_{cc}$ centered at $x$ by $B_{cc}(x,r)$. Denote the ball of radius $r$ in $d^u_{cc}$ centered at $x$ by $B^u_{cc}(x,r)$.
\end{defn}

\begin{assump}
Henceforth we assume that $\varphi:\mathbb M\to\widetilde{\mathbb M}$ is a mapping of Carnot manifolds $\mathbb M$ and $\widetilde{\mathbb M}$. We specify its smoothness below in Assumption~{\ref{assumpphi}}.
\end{assump}

\begin{notat} Hereinafter, we use the following notation. Namely, we denote:

$\bullet$ the topological dimension of $\mathbb M$ $(\widetilde{\mathbb M})$ by $N$ $(\widetilde{N})$;

$\bullet$ the Hausdorff dimension of $\mathbb M$ $(\widetilde{\mathbb M})$ by $\nu$ $(\tilde{\nu})$;

$\bullet$ the depth of $\mathbb M~(\widetilde{\mathbb M})$ by $M~(\widetilde{M})$.

In addition, we consider

$\bullet$ horizontal subbundles $H=H_1\subset T\mathbb M$ and $\widetilde{H}=\widetilde{H}_1\subset T\widetilde{\mathbb M}$ on $\mathbb M~(\widetilde{\mathbb M})$;

$\bullet$ subspaces $H\subset H_2\subset\ldots\subset H_M=T\mathbb M$ ($\widetilde{H}\subset \widetilde{H}_2\subset\ldots\subset \widetilde{H}_M=T\widetilde{\mathbb M}$) of dimensions $n=n_1=\dim H_1<\dim H_2<\ldots<\dim H_M=N$ ($\tilde{n}=\tilde{n}_1=\dim \widetilde{H}_1<\dim \widetilde{H}_2<\ldots<\dim \widetilde{H}_{\widetilde{M}}=\widetilde{N}$) respectively at every point $x\in\mathbb M~(\tilde{x}\in\widetilde{\mathbb M})$ (see Definition~{\ref{carnotmanifold}}).

We put $n_1=\dim H_1$, $\tilde{n}_1=\dim\widetilde{H}_1$, $n_k=\dim H_{k}-\dim H_{k-1}$ ($\tilde{n}_k=\dim \widetilde{H}_{k}-\dim \widetilde{H}_{k-1}$), $k=2,\ldots, M$ ($\widetilde{M}$).

Denote by $d_{cc}$ ($\widetilde{d}_{cc}$) the Carnot--Carath\'{e}odory metric in $\mathbb M$ ($\widetilde{\mathbb M}$), and by $d_{cc}^u$ ($\widetilde{d}_{cc}^w$) the Carnot--Carath\'{e}odory metric in $\mathcal G^u\mathbb M$ ($\mathcal G^w\widetilde{\mathbb M}$).
\end{notat}

\begin{thm} [{\cite{V2}}] {\it Suppose that $E\subset\mathbb M$ is an open set, and let $\varphi:\mathbb
M\to\widetilde{\mathbb M}$ be a mapping with $C^1$-smooth horizontal derivatives $X_i\varphi$ such that $X_i\varphi\in\widetilde{H}$, $i=1,\ldots, n$. Then, it is $hc$-differentiable in points of  $E$. Namely, for a point $u\in E$, there exists a horizontal homomorphism
$L_u:(\mathcal G^u\mathbb M, d_{cc}^u)\to ({\mathcal
G}^{\varphi(u)}\widetilde{\mathbb M}, \widetilde{d}_{cc}^{\varphi(u)})$ of
local Carnot groups such that it is continuous in $u$ and
$$
\widetilde{d}_{cc}(\varphi(w),L_u[w])=o(d_{cc}(u,w))\ \text{as $E\cap
{\mathcal G^u\mathbb M}\ni w\to u$}
$$}
where $o(\cdot)$ is locally uniform.
\end{thm}

\begin{rem}
Using the exponential mapping $\theta_u$, we can consider $L_u$ both as a homomorphism of local Carnot groups and as a homomorphism of Lie algebras of these local Carnot groups.
\end{rem}

\begin{notat}
Henceforth we denote the $hc$-differential $L_u$ of $\varphi$ at $u$ by the symbol $\widehat{D}\varphi(u)$.
\end{notat}

\begin{cor}[{\cite{V2}}]
Let $\varphi:\mathbb M\to\widetilde{\mathbb M}$ be a contact (i.~e., $D\varphi\langle H\rangle\subset\widetilde{H}$) $C^{1}$-mapping of Carnot manifolds (in the Riemannian sense). Then, it is continuously $hc$-differentiable everywhere on $\mathbb M$.
\end{cor}

\begin{property}
Suppose that $\varphi\in C^1(\mathbb M,\widetilde{\mathbb M})$. Then, the matrices of its differential $D\varphi(u)$ (in the bases $\{X_i(u)\}_{i=1}^N$ and
$\{\widetilde{X}_i(\varphi(u))\}_{i=1}^{\widetilde{N}}$) and of its $hc$-differential $\widehat{D}\varphi(u)$ (in the bases $\{\widehat{X}^u_i(v)\}_{i=1}^N$ and
$\{\widetilde{\widehat{X}}_i(\varphi(v))\}_{i=1}^{\widetilde{N}}$, where $v$ need not be equal to $u$) have the following structures:
\begin{equation}\label{riemdiff}
D\varphi(u)=\left(%
\begin{array}{ccccccccc}
 {\cal V}_1(u) &  * & * & * & * & * & * & \ldots & * \\
0 & {\cal V}_2(u) & * & * & * & * & * & \ldots & * \\
0 & 0 & {\cal V}_3(u) & * & * & * & * & \ldots & * \\
\vdots & \vdots & \vdots & \ddots & \ddots & \vdots & * & \ldots & * \\
0 & \ldots & 0 & \vdots & {\cal V}_{\widetilde{M}-1}(u) & * & * & \ldots & * \\
0 & \ldots & 0 & \ldots & 0 & {\cal V}_{\widetilde{M}}(u) & * & \ldots & * \\
\end{array}%
\right),
\end{equation}
\begin{equation}\label{subriemdiff}
\widehat{D}\varphi(u)=\left(%
\begin{array}{ccccccccc}
 {\cal V}_1(u) &  0 &  & \cdots & 0 & 0 & 0 & \ldots & 0 \\
0 & {\cal V}_2(u) & 0 & \cdots & \vdots & 0 & 0 & \ldots & 0 \\
0 & 0 & {\cal V}_3(u) & 0 & 0 & 0 & 0 & \ldots & 0 \\
\vdots & \vdots & \vdots & \ddots & \ddots & \vdots & 0 & \ldots & 0 \\
0 & \ldots & 0 & \vdots & {\cal V}_{\widetilde{M}-1}(u) & 0 & 0 & \ldots & 0 \\
0 & \ldots & 0 & \ldots & 0 & {\cal V}_{\widetilde{M}}(u) & 0 & \ldots & 0 \\
\end{array}%
\right),
\end{equation}
where each block $\mathcal V_i$ is of size $(n_i\times\tilde{n}_i)$, $i=1,\ldots, M$.
Note that, the blocks $\mathcal V_i$ are the same in the
matrices of $D\varphi(u)$ and $\widehat{D}\varphi(u)$.
\end{property}

\begin{assump}\label{dimtop} Throughout the article we assume that:

\begin{itemize}
\item $N\geq \widetilde{N}$;

\item If the matrix of $\widehat{D}\varphi(x)$ has full rank at at least one point then we suppose that $n\geq \tilde{n}$. Otherwise, we suppose that $n_i\geq \tilde{n}_i$, $i=1,\ldots, M$ (see Proposition {\ref{propdiff}}).
\end{itemize}
\end{assump}

\begin{notat}\label{setz} Given $\varphi\in C^1(\mathbb M,\widetilde{\mathbb M})$, denote by $Z$ the set of points $x\in\mathbb M$ with $\operatorname{rank}(D\varphi(x))<\widetilde{N}$.
\end{notat}

\begin{prop}\label{propdiff} {\bf I.} Given $\varphi\in C^1(\mathbb M,\widetilde{\mathbb M})$, consider a point  $x\in\mathbb M\setminus Z$.

 {\bf (a)} If $\widehat{D}\varphi(x)(V_1)=\widetilde{V}_1$ then
$\widehat{D}\varphi(x)(V_i)=\widetilde{V}_i$ for all $i=1,\ldots,M$,
and $\operatorname{rank}\widehat{D}\varphi(x)=\widetilde{N}$.

{\bf (b)} If $\widehat{D}\varphi(x)(V_1)\varsubsetneq\widetilde{V}_1$
then $\operatorname{rank}\widehat{D}\varphi(x)<\widetilde{N}$.

{\bf II.} At the points $x\in\mathbb M\setminus Z$ with
$\operatorname{rank}\widehat{D}\varphi(x)=\widetilde{N}$ we have

{\bf (a)} ${D}\varphi(x)(H_i)=\widetilde{H}_i$, $i=1,\ldots,\widetilde{M}$;

{\bf (b)}
$\mathcal {D}\varphi(x)(H_i/H_{i-1})=\widetilde{H}_i/\widetilde{H}_{i-1}$,
$i=2,\ldots,\widetilde{M}$. Here for every element
$\mathcal  Y\in{H}_i/{H}_{i-1}$ there exists $Y\in H_i$ such that
$${\mathcal Y}=\{Y+T: T\in H_{i-1}\}=Y+H_{i-1},
$$
and we define for $\mathcal Y$ the value $\mathcal D\varphi\langle\mathcal Y\rangle$ as the class
\begin{multline*}
\mathcal D\varphi\langle\mathcal Y\rangle=\{D\varphi\langle Y\rangle+D\varphi\langle T\rangle: Y\in H_i \text{ is fixed},\, T\in H_{i-1}\}\\
=D\varphi\langle Y\rangle+D\varphi\langle H_{i-1}\rangle.
\end{multline*}
\end{prop}

\begin{proof} {\bf I.} {\bf (a)} Fix $x\in\mathbb M\setminus Z$ such that
$\widehat{D}\varphi(x)(V_1)=\widetilde{V}_1$ and denote
$\widehat{D}\varphi(x)$ by~$L$. Verify that
$LV_2=\widetilde{V}_2$. By taking into account property (3) of Definition {\ref{carnotmanifold}} and the property $L[X,Y]=[LX,LY]$ of a group homomorphism $L$, we infer that $L(V_2)\subset\widetilde{V}_2$. Verify that $\widetilde{V}_2\subset L(V_2)$.
By property (4) of Definition~{\ref{carnotmanifold}}, for
each element $\widetilde{Y}\in\widetilde{V}_2$ we have
$\widetilde{Y}=[\widetilde{Y}_1,\widetilde{Y}_2]$, where
$\widetilde{Y}_1,\widetilde{Y}_2\in\widetilde{V}_1$. Since
$\widetilde{Y}_i=LY_i$, where $Y_i\in V_1$, $i=1,2$, it follows that
$$
\widetilde{Y}=[\widetilde{Y}_1,\widetilde{Y}_2]=[LY_1,LY_2]=L[Y_1,Y_2]=LY,
$$
where $Y\in V_2$. Thus, $\widehat{D}\varphi(x)(V_2)=\widetilde{V}_2$.

Similar arguments show that
$\widehat{D}\varphi(x)(V_i)=\widetilde{V}_i$ for all $i=3,\ldots, M$.
Thus, $\widehat{D}\varphi(x)[\mathbb G_x\mathbb M]=\mathbb
G_{\varphi(x)}\widetilde{\mathbb M}$, and $\operatorname{rank}\widehat{D}\varphi(x)=\widetilde{N}$.

{\bf (b)} Note that the image of every basis vector field under the $hc$-differential $\widehat{D}\varphi(x)$ is a vector field of the same degree or the zero vector field (i.~e., the $hc$-differential cannot decrease the
degree of each basis vector field). If
$\widehat{D}\varphi(x)(V_1)\varsubsetneq\widetilde{V}_1$ then
$\widehat{D}\varphi(x)^{-1}(\widetilde{V}_1\setminus[\widehat{D}\varphi(x)(V_1)])=\varnothing$,
and $\widehat{D}\varphi(x)[\mathbb G_x\mathbb M]\neq\mathbb
G_{\varphi(x)}\widetilde{\mathbb M}$. Consequently,
$\operatorname{rank}\widehat{D}\varphi(x)<\widetilde{N}$.

We prove claims {\bf (a)} and {\bf (b)} simultaneously.
The properties of the Riemannian differential $D\varphi(x)$ (see {\eqref{riemdiff}})
imply that $D\varphi(x)(H_k)\subset\widetilde{H}_k$ for
$k=1,\ldots,M$.
In particular, $D\varphi(x)(H_1)\subset\widetilde{H}_1$ by the definition of a contact mapping. Moreover,
$D\varphi(x)(H_1)=\widetilde{H}_1$ if and only if
$\widehat{D}\varphi(x)(V_1)=\widetilde{V}_1$.

In order to verify that
$\mathcal {D}\varphi(x)(H_2/H_{1})\subset\widetilde{H}_2/\widetilde{H}_{1}$,
take ${\mathcal Y}\in{H}_2/{H}_1$. Consider ${\cal Y}$ as a class of
sums of a fixed element $Y$ of $H_2$, which may have nonzero coordinates (in the basis $\{X_i\}_{i=1}^N$)
only with indices greater than $n_1$ and less than $n_1+n_2+1$,
and all elements of $H_1$:
$$
{\mathcal Y}=\{Y+T: Y\in H_2 \text{ is fixed}, T\in H_1\}.
$$
Then, the structure of $D\varphi(x)$ implies that the image
$$
\widetilde{\cal
Y}=\mathcal D\varphi(x)\mathcal{Y}=D\varphi(x)Y+D\varphi(x)(H)=D\varphi(x)Y+\widetilde{H}\in\widetilde{H}_2/\widetilde{H}_1
$$
consists of the vectors with only the first $\tilde{n}_1+\tilde{n}_2$
nonzero components. Moreover, the components with indexes from
$\tilde{n}_1+1$ to $\tilde{n}_2$ are the same for all vectors from $\widetilde{\cal
Y}$. (To verify that, it suffices to write the vectors $Y$ and $T$ in the basis $\{X_i\}_{i=1}^N$, and the matrix of $D\varphi$ in the bases $\{X_i\}_{i=1}^N$ and $\{\widetilde{X}_i\}_{i=1}^{\widetilde{N}}$, and consider the product of this matrix and these vectors.) Thus,
$\mathcal D\varphi(x)(H_2/H_1)\subset\widetilde{H}_2/\widetilde{H}_1$.

Verify that
$\mathcal {D}\varphi(x)(H_2/H_{1})\supset\widetilde{H}_2/\widetilde{H}_{1}$.
Take $\widetilde{\cal Y}\in\widetilde{H}_2/\widetilde{H}_1$ and
assume on the contrary that $\widetilde{\cal
Y}\notin\mathcal {D}\varphi(x)(H_2/H_{1})$. Consider a vector field
$\widetilde{Y}\in\widetilde{\cal Y}$; then,
$\widetilde{Y}=\widetilde{Y}_2+\widetilde{Y}_1$ with
$\widetilde{Y}_1\in\widetilde{H}_1$. Without loss of
generality we may assume that $\widetilde{Y}_1=0$. Since
$\operatorname{rank}\widehat{D}\varphi(x)=\widetilde{N}$, it follows that the images $D\varphi(x)(T_x\mathbb M)$
and $(D\varphi)_{\widetilde{M}}(x)(T_x\mathbb M)$ coincide,
as in both cases we obtain the whole $T_{\varphi(x)}\widetilde{\mathbb M}$.
Here the linear operator $(D\varphi)_{\widetilde{M}}(x)$ acts on $T_x\mathbb M$ (instead of $\mathbb G_x\mathbb M$), and its matrix in the bases $\{X_i(x)\}_{i=1}^N$ and $\{\widetilde{X}_i(\varphi(x))\}_{i=1}^{\widetilde{N}}$ equals that of
$\widehat{D}\varphi(x)$ (written in the bases $\{\widehat{X}^x_i(x)\}_{i=1}^N$ and $\{\widehat{\widetilde{X}}{}^{\varphi(x)}_i(\varphi(x))\}_{i=1}^{\widetilde{N}}$). Denote by $L$ the matrix of this operator.
Then, there exists a vector $Y_0$ with $LY_0=\widetilde{Y}$.
Without loss of generality we may assume that $Y_0$ has at most $n_2$ nonzero components with indices from $n_1+1$ to
$n_1+n_2$. Consequently,
$$
\widetilde{Y}=LY_0=D\varphi(x)Y_0+[L-{D}\varphi(x)]Y_0,
$$
where $D\varphi(x)Y_0\in\widetilde{H}_2$ and
$[L-{D}\varphi(x)]Y_0\in\widetilde{H}_1$. Since
$D\varphi(x)(H_1)=\widetilde{H}_1$, there exists $Y_1\in{D\varphi(x)}^{-1}([L-{D}\varphi(x)]Y_0)\cap H_1$. Put $Y=Y_0+Y_1$,
and then $D\varphi(x)Y=\widetilde{Y}$ and $Y\in H_2$.
Thus, $\mathcal{D}\varphi(x)(H_2/H_{1})=\widetilde{H}_2/\widetilde{H}_{1}$
and, moreover, $D\varphi(x)(H_2)=\widetilde{H}_2$.

Similarly, we can show that
$\mathcal D\varphi(x)(H_i/H_{i-1})=\widetilde{H}_i/\widetilde{H}_{i-1}$ and
$D\varphi(x)(H_i)=\widetilde{H}_i$, $i=3,\ldots,\widetilde{M}$. The
proposition follows.
\end{proof}

Claim {\bf I} implies

\begin{cor} The existence of $x\in\mathbb M$ with $\operatorname{rank}\widehat{D}\varphi(x)=\widetilde{N}$ implies that
$n_i\geq\tilde{n}_i$, $i=1,\ldots,M$.
\end{cor}

\begin{defn} The set
$$
\chi=\{x\in\mathbb M\setminus Z:\operatorname{rank}\widehat{D}\varphi(x)<\widetilde{N}\}
$$
is called the {\it characteristic} set. The points of $\chi$ are
called {\it characteristic points}.

For $t\in\widetilde{\mathbb M}$, denote the intersection $\varphi^{-1}(t)\cap\chi$ by~$\chi_t$.
\end{defn}

\begin{rem}
For $\widetilde{\mathbb M}=\mathbb R$ this definition coincides with the definition of the characteristic set given by P.~Pansu~{\cite{P}} and J. Heinonen~{\cite{H}}: in this case the condition $\operatorname{rank}\widehat{D}\varphi(x)<1$ implies that all the horizontal derivatives $X_i\varphi(x)$, $i=1,\ldots, n$, vanish at $x$, and vice versa.
\end{rem}

\begin{rem}
Proposition {\ref{propdiff}} explains the second part of Assumption {\ref{dimtop}}: if $\mathbb M\neq\chi$ then we do not need to assume that $n_i\geq\tilde{n}_i$ for $i\geq 2$.
\end{rem}

\begin{assump}\label{assumpphi} If $\mathcal H^{\nu}(\chi)=\mathcal H^{N}(\chi)=0$ then we assume that $\varphi\in C^1(\mathbb M, \widetilde{\mathbb M})$ and $X_i\in C^{2}(\mathbb M)$, $i=1,\ldots, N$ (this condition is sufficient for establishing the $hc$-differentiability of $\varphi$); otherwise, we assume that $\varphi\in C^{M+1}(\mathbb M, \widetilde{\mathbb M})$ and $X_i\in C^{M+1}(\mathbb M)$, $i=1,\ldots, N$.
\end{assump}

\begin{defn}The set
$$
\mathbb D=\{x\in\mathbb M:\operatorname{rank}\widehat{D}\varphi(x)=\widetilde{N}\}
$$
is called the {\it regular} set. If $x\in\mathbb D$ then we say that $x$ is a {\it regular} point.
\end{defn}

\begin{lem}\label{sum} {\bf I.} For the set
\begin{multline*}
\zeta=\Bigl\{x\in\mathbb M\setminus Z:\ \exists\{X_{i_1},\ldots, X_{i_{\widetilde{N}}}\}\\
\bigl(\operatorname{rank}([X_{i_j}\varphi](x))_{j=1}^{\widetilde{N}}=\widetilde{N}\bigl)\Rightarrow\Bigl(\sum\limits_{j=1}^{\widetilde{N}}\operatorname{deg}
X_{i_j}<\tilde{\nu}\Bigr)\Bigr\}
\end{multline*}
we have $\zeta=\varnothing$ $($see Notation~${\ref{setz}}$ for the description of~$Z)$.

{\bf II.} If there exists a family $\{X_{i_1},\ldots,X_{i_{\widetilde{N}}}\}$ of vector fields with the properties
$\operatorname{rank}([X_{i_j}\varphi](x))=\widetilde{N}$ and $\sum\limits_{j=1}^{\widetilde{N}}\operatorname{deg}
X_{i_j}=\tilde{\nu}$, then we have $\operatorname{deg} X_{i_j}\leq \widetilde{M}$,
$j=1,\ldots,\widetilde{N}$.
\end{lem}

\begin{proof} Fix $x\in\mathbb M$ and consider the matrix of $D\varphi(x)$:
$$
D\varphi(x)=\left(%
\begin{array}{ccccccccc}
 {\cal V}_1(u) &  * & * & * & * & * & * & \ldots & * \\
0 & {\cal V}_2(u) & * & * & * & * & * & \ldots & * \\
0 & 0 & {\cal V}_3(u) & * & * & * & * & \ldots & * \\
\vdots & \vdots & \vdots & \ddots & \ddots & \vdots & * & \ldots & * \\
0 & \ldots & 0 & \vdots & {\cal V}_{\widetilde{M}-1}(u) & * & * & \ldots & * \\
0 & \ldots & 0 & \ldots & 0 & {\cal V}_{\widetilde{M}}(u) & * & \ldots & * \\
\end{array}%
\right).
$$
Choose $\widetilde{N}$
linearly independent columns with the minimal possible sum of the
corresponding degrees (we consider $\operatorname{deg}X_j$ as the degree of column $j$, $j=1, \ldots, N$).

To this end, we must choose the maximal possible quantity of
vectors from the blocks corresponding to the minimal degrees. In the
first block, we can choose at most $\tilde{n}_1$ linearly
independent vectors. Next, take columns from $n_1+1$ to $n_1+n_2$ and the
corresponding ``diagonal'' block. In this block, we can
choose at most $\tilde{n}_2$ linearly independent $\tilde{n}_2$-dimensional
elements.

On assuming that there are more than $\tilde{n}_2$
linearly independent columns, we obtain a contradiction.
Indeed, since the ``diagonal'' block is an $(\tilde{n}_2\times
n_2)$-matrix, there exists an elementary transformation reducing at least $n_2-\tilde{n}_2$ of its columns to zero. Apply this
transformation to the columns of the matrix of $D\varphi(x)$. Then, this block of size $(\tilde{n}_1+\tilde{n}_2)\times(n_1+n_2)$
includes $n_1+n_2-\tilde{n}_2$ columns of dimension
$\tilde{n}_1$ (more exactly, these column vectors belong to $\mathbb R^{\tilde{n}_1}\times 0^{\tilde{n}_2}$). Recall that the maximal number of linearly independent columns is $\tilde{n}_1$, and we have already chosen them in the first ``diagonal'' block. Suppose that there are less than $\tilde{n}_1+\tilde{n}_2$ linearly independent columns. Then, since $\operatorname{rank}D\varphi(x)=\widetilde{N}$, a ``missing'' column can be ``replaced'' by a column of a higher degree. Thus, we can obtain the minimal possible sum of degrees if we have only
$n_1-\tilde{n}_1+n_2-\tilde{n}_2$ linearly dependent among the first $n_1+n_2$ columns. Therefore, the corresponding sum of degrees equals $\tilde{n}_1+2\tilde{n}_2$.

Applying further the same arguments to the degrees $3,\ldots,\widetilde{M}$,
we conclude that the minimal possible sum of degrees of linearly
independent vector fields $\{X_{i_j}\varphi\}_{j=1}^{\widetilde{N}}$ is equal
to $\tilde{\nu}$. Thus, claim {\bf I} is proved.

Claim {\bf II}
follows since, firstly, we can obtain the sum
equal to $\tilde{\nu}$ only by considering the first $\widetilde{M}$ blocks,
and, secondly, if  we have less than $\tilde{n}_k$ linearly
independent vector fields on step $k\leq \widetilde{M}$ then the
sum of degrees corresponding to the resulting collection is
strictly greater than $\tilde{\nu}$.
\end{proof}

\begin{thm}\label{structure} {\bf I.} The characteristic set $\chi$
coincides with
\begin{multline}\label{chi}
\Bigl\{x\in\mathbb M\setminus Z:\ \forall\{X_{i_1},\ldots, X_{i_{\widetilde{N}}}\}\\
\bigl(\operatorname{rank}([X_{i_j}\varphi](x))_{j=1}^{\widetilde{N}}=\widetilde{N}\bigl)\Rightarrow\Bigl(\sum\limits_{j=1}^{\widetilde{N}}\operatorname{deg}
X_{i_j}>\tilde{\nu}\Bigr)\Bigr\}.
\end{multline}
{\bf II.} The regular set $\mathbb D$ coincides with
\begin{multline}\label{reg}
\Bigl\{x\in\mathbb M\setminus Z:\ \exists\{X_{i_1},\ldots, X_{i_{\widetilde{N}}}\}\\
\bigl(\operatorname{rank}([X_{i_j}\varphi](x))_{j=1}^{\widetilde{N}}=\widetilde{N}\bigl)\Rightarrow\Bigl(\sum\limits_{j=1}^{\widetilde{N}}\operatorname{deg}
X_{i_j}=\tilde{\nu}\Bigr)\Bigr\}.
\end{multline}
\end{thm}

\begin{proof} {\bf I.} {\bf (a)} Denote the set in {\eqref{chi}} by $A$,
and verify that $A\subset\chi$. Consider a point $x\in A$.  By
{\eqref{chi}}, every $\widetilde{N}$ columns of the matrix of
$D\varphi(x)$ with indices $i_1,\ldots, i_{\widetilde{N}}$ corresponding to a collection $\{X_{i_1},\ldots,
X_{i_{\widetilde{N}}}\}$ with $\sum\limits_{j=1}^{\widetilde{N}}\operatorname{deg} X_{i_j}=\tilde{\nu}$
are linearly dependent. Assume on the contrary that the rank
of the matrix of $\widehat{D}\varphi(x)$ equals $\widetilde{N}$, and
consequently, there exist $\widetilde{N}$ linearly independent columns in the matrix of $\widehat{D}\varphi(x)$. The matrix of $\widehat{D}\varphi(x)$ has a block structure, where
the block $k$ is an $(\tilde{n}_k\times n_k)$-matrix. Thus, in each
block, only $\tilde{n}_k$ columns can be linearly
independent. Consequently, the sum of the degrees of the vector fields
corresponding to these linearly independent columns equals $\tilde{\nu}$.

The relation between the matrices of $D\varphi(x)$ and
$\widehat{D}\varphi(x)$ implies that the corresponding columns of the matrix of
$D\varphi(x)$ are also linearly independent, and the sum of the degrees of
the corresponding vector fields is equal to $\tilde{\nu}$. Thus, we arrive at a contradiction.

The argument above implies that $A\subset\chi$.

{\bf (b)} Verify that $\chi\subset A$. Consider $x\in\chi$. Since
$\operatorname{rank}\widehat{D}\varphi(x)<\widetilde{N}$, it follows that every $\widetilde{N}$ columns are
linearly dependent. Our goal is to show that if we take the columns of the matrix of $D\varphi(x)$ with indices $i_1,\ldots, i_{\widetilde{N}}$ such that
$$\sum\limits_{j=1}^{\widetilde{N}}\operatorname{deg}
X_{i_j}\leq\tilde{\nu}$$
then
$$\operatorname{rank}([X_{i_j}\varphi](x))_{j=1}^{\widetilde{N}}<\widetilde{N}.
$$ In view of claim {\bf I} of Lemma~{\ref{sum}}, it suffices to consider the columns with indices $i_1,\ldots, i_{\widetilde{N}}$ satisfying
$$\sum\limits_{j=1}^{\widetilde{N}}\operatorname{deg}
X_{i_j}=\tilde{\nu}.
$$ Take $\widetilde{N}$ columns
$c_{i_1},\ldots,c_{i_{\widetilde{N}}}$ of the matrix of $\widehat{D}\varphi(x)$ corresponding to some vectors
$X_{i_1},\ldots,X_{i_{\widetilde{N}}}$ with $\sum\limits_{j=1}^{\widetilde{N}}\operatorname{deg}
X_{i_j}=\tilde{\nu}$. Since
$$\operatorname{rank}(\{c_{i_1},\ldots,c_{i_{\widetilde{N}}}\})<\widetilde{N},
$$
there exists a
transformation $\tau$ of the  matrix of $\widehat{D}\varphi(x)$ taking at least one of $c_{i_j}$ to a zero column, and preserving the block structure of the initial
matrix. Denote this column by
$c_{i_{j_0}}$. Here we assume that $j_0$ is the minimal number with this property. Put $k=\operatorname{deg} X_{i_{j_0}}-1$. Because of the block
structure of the matrix $\widehat{D}\varphi(x)$, the numbers of
nonzero entries of $c_{i_{j_0}}$ are at least $\dim
\widetilde{H}_k+1$, and at most $\dim \widetilde{H}_{k+1}$. Without loss of generality we may assume that $i_{j_{0}}=\dim H_k+1$.
Since we cannot transform all preceding columns into zero columns, it follows that $j_0=\dim \widetilde{H}_k+1$.

Apply the same transformation $\tau$ to the matrix of $D\varphi(x)$, and consider the images
of $j_0$ columns with indices $i_1,\ldots,i_{j_0}$. Note
that we may regard the image of column ${i_{j_0}}=\dim H_k+1$ as an element of
$\mathbb R^{\dim \widetilde{H}_k}$. Taking the structure of $D\varphi(x)$ into
account, we have $j_0=\dim \widetilde{H}_k+1$ vectors (columns) belonging to
the space $\mathbb R^{\dim \widetilde{H}_k}$ (because the components with
indices greater than $\dim \widetilde{H}_k$ vanish). Thus, the rank of
this collection equals $\dim \widetilde{H}_k<j_0$. Consequently, the rank
of $\{X_{i_1}\varphi,\ldots,X_{i_{\widetilde{N}}}\varphi\}$ is strictly less
than $\widetilde{N}$. Since the collections of $\widetilde{N}$ vectors with the
sum of degrees equal to $\tilde{\nu}$ appear only in
the matrix of $\widehat{D}\varphi(x)$ (see Lemma {\ref{sum}}), we see
that if the sum of the degrees is equal to $\tilde{\nu}$, then
the rank is strictly less than $\widetilde{N}$.  By taking Lemma {\ref{sum}} and
the fact that $\operatorname{rank}(D\varphi(x))=\widetilde{N}$ into account, we see that
$\chi\subset A$.

{\bf II.} {\bf (a)} Denote the set in {\eqref{reg}} by $B$. Assume the contrary and take $x\in B$ with
$\operatorname{rank}(\widehat{D}\varphi(x))<\widetilde{N}$. Step {\bf I} yields $x\in
A$, and thus, we obtain a contradiction since $A\cap B=\varnothing$. Consequently,
$B\subset\mathbb D$.

{\bf (b)} Take $x\in\mathbb D$. Since $\operatorname{rank}\widehat{D}\varphi(x)=\widetilde{N}$, it follows that
$\operatorname{rank}{D}\varphi(x)=\widetilde{N}$ and $x\notin Z$. Take an arbitrary collection
of $\widetilde{N}$ linearly independent columns of the matrix of
$\widehat{D}\varphi(x)$. Consequently, the corresponding columns of the matrix of
$D\varphi(x)$ are linearly independent as well, and $x\in B$. Thus,
$\mathbb D\subset B$.

The theorem follows.
\end{proof}

\section{Properties of Level Sets}

In this section, we assume that $x\in\mathbb M\setminus Z$.

First of all, we introduce a new metric, which is equivalent to the initial one and simplifies our computations.

\begin{defn} Let $\mathbb M$ be a Carnot manifold of topological dimension
$N$ and depth $M$, and put
$x=\exp\Bigl(\sum\limits_{i=1}^{N}x_i X_i\Bigr)(g)$. Define the
distance $d_2(x,g)$ as follows:
\begin{multline*}
d_2(x,g)=\max\Bigl\{\Bigl(\sum\limits_{j=1}^{n_1}|x_j|^2\Bigr)^{\frac{1}{2}},\\
\Bigl(\sum\limits_{j=n_1+1}^{n_1+n_2}|x_j|^2\Bigr)^{\frac{1}{2\cdot\operatorname{deg}
X_{n_1+1}}},\ldots,\Bigl(\sum\limits_{j=N-n_{M}+1}^{N}|x_j|^2\Bigr)^{\frac{1}{2\cdot\operatorname{deg}
X_{N}}}\Bigr\}.
\end{multline*}

A similar metric $d_2^u$ is introduced on the local Carnot
group $\mathcal G^u\mathbb M$.

The set $\{y\in\mathbb M: d_2(y,x)<r\}$ is called the ball of radius $r>0$ centered at $x$ and denoted by $\operatorname{Box}_2(x,r)$. Similarly, $\operatorname{Box}_2^u(x,r)$ stands for the ball in $d_2^u$ of radius $r>0$  centered at $x$.
\end{defn}

\begin{rem}\label{box2} The preimage of $\operatorname{Box}_2(x,r)$ in the metric
$d_2$ under the mapping $\theta_x$ equals
$$
\operatorname{Box}_2(0,r)=B_2^{n_1}(x,r)\times B_2^{n_2}(x,r^2)\times\ldots\times
B_2^{n_{M}}(x,r^{M}),
$$
where $B_2^{n_i}$ is a Euclidean ball of dimension $n_i$, $i=1,\ldots,M$.

Observe that
$\operatorname{Box}_2^u(u,r)=\operatorname{Box}_2(u,r)$
in the quasimetric $d_2^u$ since
\begin{equation}\label{eqexp}
\exp\Bigl(\sum\limits_{i=1}^{N}x_i \widehat{X}^u_i\Bigr)(u)=\exp\Bigl(\sum\limits_{i=1}^{N}x_i X_i\Bigr)(u)
\end{equation}
for all collections $\{x_i\}_{i=1}^N$ such that both parts of {\eqref{eqexp}} make sense {\cite{V2}}.
\end{rem}

The following proposition is useful for proving the main results.

\begin{prop}\label{prop21vg}  Let $\mathbb M$ be a Carnot manifold of topological dimension
$N$ and depth $M$. Given a sufficiently small compact domain
$U\Subset\mathbb M$, there exist
positive constants $C>0$ and $r_0>0$ depending on $U$, $M$, and
$N$ such that all points $u$ and $v$ of $U$ satisfy
$$
\bigcup\limits_{x\in\operatorname{Box}_2^u (v,r)} \operatorname{Box}_2^u (x,\xi)\subseteq
\operatorname{Box}_2^u (v,r+C\xi),\quad 0<\xi,\,  r\leq r_0.
$$
\end{prop}
\begin{proof} The proof follows the scheme of proof of the similar lemma for
boxes in the metric $d^u_{\infty}$ of {\cite{VK, vk_birk}}.

Put $x=\exp\Bigl(\sum\limits_{i=1}^Nx_i\widehat{X}^u_i\Bigr)(v)$,
$d_2^u(v,x)\leq r$, and
$z=\exp\Bigl(\sum\limits_{i=1}^Nz_i\widehat{X}^u_i\Bigr)(x)$,
$d_2^u(x,z)\leq\xi$. Estimate the distance $d_2^u(v,z)$
applying group operation to points $x$ and $z$. Namely, estimate the coefficients $\{\zeta_i\}_{i=1}^N$ satisfying
$z=\exp\Bigl(\sum\limits_{i=1}^N\zeta_i\widehat{X}^u_i\Bigr)(v)$.

{\sc Case of $\operatorname{deg} X_i=1$.} We have
$$
\sum\limits_{i=1}^{n}(\zeta_i)^2=\sum\limits_{i=1}^{n}(x_i+z_i)^{2}
=\sum\limits_{i=1}^{n}(x_i)^{2}+\sum\limits_{i=1}^{n}(z_i)^{2}+2\sum\limits_{i=1}^{n}(x_iz_i)\leq
(r+\xi)^{2\operatorname{deg} X_i}.
$$

{\sc Case of $\operatorname{deg} X_i=2$.} We have
\begin{multline*}
\sum\limits_{i=n+1}^{n+n_2}(\zeta_i)^2\leq
\sum\limits_{i=n+1}^{n+n_2}\Bigl(x_i+z_i
+\sum\limits_{\substack{|e_l+e_j|_h=2,\\l<j}}\widehat{F}^i_{e_l,e_j}(u)(x_lz_j-z_lx_j)\Bigr)^2\\
\leq r^4+\xi^4+2r^2\xi^2
+\sum\limits_{i=n+1}^{n+n_2}\Bigl(\sum\limits_{\substack{|e_l+e_j|_h=2,\\l<j}}
\widehat{F}^i_{e_l,e_j}(u)(x_lz_j-z_lx_j)\Bigr)^2\\
+2\sum\limits_{i=n+1}^{n+n_2}\Bigl((x_i+z_i)\sum\limits_{\substack{|e_l+e_j|_h=2,\\l<j}}
\widehat{F}^i_{e_l,e_j}(u)(x_lz_j-z_lx_j)\Bigr)\\
\leq r^4+\xi^4+c_i(u)r^2\xi^2+b_i(u)r\xi(r+\xi)^2\\
\leq(r+a_2(u)\xi)^4 =(r+{a_2(u)}\xi)^{2\operatorname{deg} X_i},
\end{multline*}
where $b_i(u)$ and $c_i(u)$ are linear combinations of the functions $\{\widehat{F}^i_{e_l,e_j}(u)\}_{l,j}$ on assuming that $|x_lz_j-z_lx_j|=2r\xi$ and $|x_i+z_i|=r+\xi$ for all $i,l,j$. They are continuous with respect to $u\in U$. We can represent each sum $r^4+\xi^4+c_i(u)r^2\xi^2+b_i(u)r\xi(r+\xi)^2$ as $(r+d_i(u)\xi)^4$, where $d_i(u)$ depends on $b_i(u)$ and $c_i(u)$. Put $a_2(u)=\max\limits_{i:\,\operatorname{deg}X_i=2}d_i(u)$ and assume without loss of generality that $a_2(u)\geq 1$.

{\sc Case of $\operatorname{deg} X_i=k>2$.} Denote the sum $n+\sum\limits_{j=2}^{k}n_j$ by $S_k$. Then, as in the previous case
of $\operatorname{deg} X_i=2$, we obtain
\begin{multline*}
\sum\limits_{i=S_{k-1}+1}^{S_k}(\zeta_i)^2
\leq
\sum\limits_{i=S_{k-1}+1}^{S_k}
\Bigl(x_i+z_i+\sum\limits_{\substack{|\mu+\beta|_h=k,\\
\mu>0,\beta>0}}|\widehat{F}^i_{\mu,\beta}(u)|x^\mu\cdot
z^\beta\Bigr)^2\\
\leq (r+a_k(u)\xi)^{2k}=(r+a_k(u)\xi)^{2\operatorname{deg} X_i}.
\end{multline*}
Here we use the property
$$
\sum\limits_{\substack{|\mu+\beta|_h=k,\\
\mu>0,\beta>0}}|\widehat{F}^i_{\mu,\beta}(u)|x^\mu\cdot
z^\beta\leq\sum\limits_{\substack{|\mu+\beta|_h=k,\\
\mu>0,\beta>0}}c^i_{\mu,\beta}(u)r^{|\mu|_h}\cdot \xi^{|\beta|_h},
$$
and define each function $a_k(u)$ in the similar way as $a_2(u)$ in the case $\operatorname{deg}X_i=2$.
We also assume without loss of generality that
$a_k(u),c^i_{\mu,\beta}(u)\geq 1$. Put
$a(u)=\max\limits_ia_i(u)$. The estimates above yield
\begin{multline*}
d_2^u(v,x)=\max\Bigl\{\Bigl(\sum\limits_{j=1}^{n}|\zeta_j|^2\Bigr)^{\frac{1}{2}},\\
\Bigl(\sum\limits_{j=n+1}^{n+n_2}|\zeta_j|^2\Bigr)^{\frac{1}{2\cdot\operatorname{deg}
X_{n+1}}},\ldots,\Bigl(\sum\limits_{j=N-n_{M}+1}^{N}|\zeta_j|^2\Bigr)^{\frac{1}{2\cdot\operatorname{deg}
X_{N}}}\Bigr\}\\
\leq \max\limits_i\{(r+a_i(u)\xi)^{\frac{\operatorname{deg} X_i}{\operatorname{deg} X_i}}\}\leq
r+a(u)\xi.
\end{multline*}
Since all $a_i(u)$ are continuous with respect to $u$, we may
choose sufficiently large $C<\infty$ with $a(u)\leq C$ for all $u$
belonging to the given compact domain $U\Subset\mathbb M$. The lemma follows.
\end{proof}

To prove the main theorems, we need a convenient quasimetric equivalent to the Riemannian metric.

\begin{defn}
Given
$$
y=\exp\Bigl(\sum\limits_{i=1}^Ny_iX_i\Bigr)(x),
$$
put $\rho(y,x)=\max\limits_{i=1,\ldots, N}\{|y_i|\}$.
\end{defn}

\begin{notat} In Theorems {\ref{tang_meas}} and
{\ref{Leb_meas}}, we establish some local results for a fixed point
$x$. From now on we use the auxiliary mapping $\psi=\varphi\circ\theta_x$.
\end{notat}

\begin{notat} Put
\begin{multline*}
\nu_0(x)=\min\Bigl\{\nu: \exists\{X_{i_1},\ldots, X_{i_{\widetilde{N}}}\}\\
\bigl(\operatorname{rank}([X_{i_j}\varphi](x))_{j=1}^{\widetilde{N}}=\widetilde{N}\bigl)\Rightarrow\Bigl(\sum\limits_{j=1}^{\widetilde{N}}\operatorname{deg}
X_{i_j}=\nu\Bigr)\Bigr\}.
\end{multline*}
\end{notat}

It is clear that ${\nu_0}|_{\chi}>\tilde{\nu}$ and
${\nu_0}|_{\mathbb D}=\tilde{\nu}$.

\begin{thm}\label{tang_meas} Fix $x\in\varphi^{-1}(t)$.
Then, in a neighborhood in $\mathbb R^N$, the
${\cal H}^{N-\widetilde{N}}$-measure of
$T_0[\psi^{-1}(t)]\cap\operatorname{Box}_2(0,r)$ $($see Remark~${\ref{box2}})$ is equal to
\begin{equation}\label{constc}
Cr^{\nu-\nu_0(x)}(1+o(1))
\end{equation}
where $C$ is independent of $r$, and $o(1)\to0$ as $r\to0$.
\end{thm}

\begin{rem}\label{tangplane}
We emphasize that the mapping $\psi=\varphi\circ\theta_x$ acts on a neighborhood of the origin in $\mathbb R^N$. Consequently, the tangent plane to the level set $\psi^{-1}(t)$ lies in $\mathbb R^N$, and the intersection $T_0[\psi^{-1}(t)]\cap\operatorname{Box}_2(0,r)$ is well-defined.
\end{rem}

\begin{proof}[Proof of Theorem {\ref{tang_meas}}] We split the proof into 6
steps. On step~{\bf I} we
choose a suitable basis $\{w_j\}_{j=1}^N$ for the tangent space $\mathcal T=T_0[\psi^{-1}(t)]$
to the level set. On step~{\bf II} we define two
projections of the basis vectors in $T_0[\psi^{-1}(t)]$. In
particular, the first projection $\pi$ assigns to each basis
vector $w_j$ some vector $p_j$ of the same degree, and the second
projection assigns
a standard vector $\frac{\pi_j}{|\pi_j|}$ in $\{e_1,\ldots,e_{N}\}$ to each basis vector $w_j$ for $T_0[\psi^{-1}(t)]$.
On step~{\bf III} we show that
$\operatorname{rank}(\sigma_i\psi)_{i=1}^{\widetilde{N}}=\widetilde{N}$ (henceforth $\sigma_i\psi$ stands for the action of the vector $\sigma_i$ on $\psi$, $i=1,\ldots, N$), where
$$
\{\sigma_1,\ldots,\sigma_{\widetilde{N}}\}=\{e_1,\ldots,e_{N}\}\setminus\Bigl\{\frac{\pi_1}{|\pi_1|},\ldots,
\frac{\pi_{N-\widetilde{N}}}{|\pi_{N-\widetilde{N}}|}\Bigr\},
$$
and on step~{\bf IV} we prove that the sum of the degrees
of $\sigma_i$, $i=1,\ldots,\widetilde{N}$, coincides with $\nu_0(x)$.
Consequently, the sum of the degrees of $\pi_j$, $j=1,\ldots,N-\widetilde{N}$,
equals $\nu-\nu_0(x)$. Further, on step~{\bf V}
we deduce that the Lebesgue measure of
$\operatorname{Box}_2(0,r)\cap\operatorname{span}\{p_1,\ldots,p_{N-\widetilde{N}}\}$ equals
$Cr^{\nu-\nu_0(x)}$, where $C$ is independent of $r$. Finally, on
step~{\bf VI} we prove that the length of $\mathbb Rw_j\cap\operatorname{Box}_2(0,r)$ equals $O(r^{k(j)})$ for
sufficiently small $r>0$, and applying this result we show
that $\pi(\mathcal T\cap\operatorname{Box}_2(0,r))$ coincides with the
$o(r)$-neighborhood of $\mathcal S\cap\operatorname{Box}_2(0,r)$ in $\mathcal S$, where $\mathcal S=\operatorname{span}\{p_1,\ldots,p_{N-\widetilde{N}}\}$ and $o(r)$ is
taken with respect to the metric $d_2$. The theorem follows from
the last result.

{\bf Step I.} Consider the normal coordinates at the point~$x$.
Recall that $\mathcal T=T_0[\psi^{-1}(t)]$. Choose an arbitrary basis in $\mathcal T$, and write it as a
matrix in which the basis vectors are written as rows.

{\bf (i)} By elementary row transformations, reduce this matrix to
\begin{equation*}
\Delta=\left(%
\begin{array}{cccccccccc}
  * & \ldots & * & 0 &  &  & \ldots &  &  & 0 \\
  * & \ldots & * & * & 0 & \ldots & 0 &  & \ldots & 0 \\
  * &  & \ldots &  & * & * & 0 & 0 & \ldots & 0 \\
  \vdots  &  & \vdots & &  &  & \ddots & \ddots &  & 0 \\
  * & & & \ldots &  &  &  & * & 0 & 0 \\
  * & & & \ldots &  &  &  & * & * & * \\
\end{array}%
\right),
\end{equation*}
where the upper right triangle consists of zeroes, and the last nonzero entries of the rows appear in distinct columns.

{\bf (ii)} In the matrix $\Delta$ there is the following natural
``grading'' of the columns corresponding to the grading of the
Lie algebra $V$ of vector fields on the local tangent cone.
Split the columns of $\Delta$ into $M$ blocks ${\cal
B}_1,\ldots,{\cal B}_{M}$ such that ${\cal
B}_k$ consists of columns from $\dim H_{k-1}+1$
to $\dim H_k$:
$$
\begin{array}{cccccccccc}
  \quad\quad {}_1 & \ldots & {}_{\dim H_1} & {}_{\dim H_1+1} & \ldots & {}_{\dim H_2} & \ldots\ldots\ldots & {}_{\dim H_{M-1}+1} & \ldots & {}_N\quad\quad \\
  \Bigl( & \mathcal B_1 & & & \mathcal B_2 &  & \ldots\ldots\ldots &  & \mathcal B_M & \Bigr). \\
\end{array}%
$$

Next, there is also a ``grading'' of rows of $\Delta$. There are
$M$ blocks ${\cal A}_{1},\ldots,{\cal A}_{M}$. Here ${\cal A}_l$ consists of the rows whose
last nonzero element appears in a column with index in $[\dim H_{l-1}+1,\dim H_l]$:
$$
\left(%
\begin{array}{c}
\mathcal A_1\\
\\
\mathcal A_2\\
\vdots\\
\\
\mathcal A_M
\end{array}%
\right)
=\left(%
\begin{array}{cccccccccc}
 {}^1 & \cdots & {}^{\dim H_1} & {}^{\dim H_1+1} & \cdots & {}^{\dim H_2} & \cdots & {}^{\dim H_{M-1}+1} & \cdots & {}^N\\
  & \ldots & * & 0 &  &  & \ldots &  &  & 0 \\
  * & \ldots & * & * & 0 & \ldots & 0 &  & \ldots & 0 \\
  * &  & \ldots &  & * & * & 0 & 0 & \ldots & 0 \\
  \vdots  &  & \vdots & &  &  & \ddots & \ddots &  & 0 \\
  * & & & \ldots &  &  &  & * & 0 & 0 \\
  * & & & \ldots &  &  &  & * & * & * \\
\end{array}%
\right).
$$
Some of these blocks can be empty.

{\bf (iii)} For each $k=1,\ldots,M$, put ${\cal V}_{k}={\cal
B}_{k}\cap \mathcal A_k$
$$
\left(%
\begin{array}{cccccccccc}
  & {\cal V}_1  & 0 & 0 &  &  & \ldots &  &  & 0 \\
*  &*  &  & {\cal V}_2 & 0 & \ldots & 0 &  & \ldots & 0 \\
  * &  & * &  & * &  & 0 & 0 & \ldots & 0 \\
  \vdots  & \mathcal B_1\setminus{\cal V}_1 & \vdots & \mathcal B_2\setminus{\cal V}_2 & * &  & \ddots & \ddots &  & 0 \\
  * & &  &  & \vdots &  & \vdots & * & {\cal V}_{M-1} & 0 \\
  * & & * &  & * &  & * &  & \mathcal B_{M-1}\setminus{\cal V}_{M-1} &{\cal V}_{M} \\
\end{array}%
\right).
$$

Further, we transform the blocks $\mathcal B_k\setminus\mathcal
V_k$, $k=1,\ldots,M$.  Define the transformation of $\Delta$ by
induction. For $k=M$, we have nothing to transform; thus, the base of induction holds.

Suppose that we have transformed the
blocks $\mathcal B_k\setminus\mathcal V_k$, $k=l+1,\ldots,M$,
$l\leq M-1$, and assume that $\mathcal V_l\neq\varnothing$
(otherwise, we have nothing to transform). Replace the blocks
${\cal A}_{l+1},\ldots,{\cal A}_{M}$ by the projections of their
row vectors onto  $(\operatorname{span}\{\mathcal
A_l\})^{\bot}\cap\operatorname{span}\{{\cal A}_l,\ldots,{\cal A}_{M}\}$. Roughly speaking, we remove the part collinear to  $\mathcal A_l$ from the row vectors of ${\cal A}_{l+1},\ldots,{\cal A}_{M}$. This projection preserves the blocks $\mathcal
B_k\setminus\mathcal V_k$, $k=l+1,\ldots,M$, $l\leq M-1$, because of the ``triangular'' structure of $\Delta$.

Moreover, the rows of $\mathcal B_{k}\setminus{\cal V}_{k}$
are orthogonal to $\mathcal V_k$ with respect to the classical inner product.

Denote the resulting vectors by $w_1,\ldots, w_{N-\widetilde{N}}$:
\begin{equation}\label{matw}
W=
\left(%
\begin{array}{cccccccccc}
  & {\cal V}_1  & 0 & 0 &  &  & \ldots &  &  & 0 \\
*  &*  &  & {\cal V}_2 & 0 & \ldots & 0 &  & \ldots & 0 \\
  * &  & * &  & * &  & 0 & 0 & \ldots & 0 \\
  \vdots  & {\cal V}_1^{\bot} & \vdots & {\cal V}_2^{\bot} & * &  & \ddots & \ddots &  & 0 \\
  * & &  &  & \vdots &  & \vdots & * & {\cal V}_{M-1} & 0 \\
  * & & * &  & * &  & * &  & {\cal V}_{M-1}^{\bot} &{\cal V}_{M} \\
\end{array}%
\right)
=\left(%
\begin{array}{c}
  w_1 \\
   \\
  \vdots \\
  \vdots \\
   \\
  w_{N-\widetilde{N}} \\
\end{array}%
\right).
\end{equation}

Thus, we have constructed a ``suitable'' basis for $\mathcal T$.

{\bf Step II.} For each $j=1,\ldots, N-\widetilde{N}$, define the number
$k(j)$ as follows. Let $l(j)$ be the row index of the last
nonzero element of $w_j$ in {\eqref{matw}}, and put $k(j)=\operatorname{deg}
X_{l(j)}$.

Project each $w_j$ onto $\mathcal V_{k(j)}$ by taking the last
nonzero coordinates belonging to $\mathcal V_{k(j)}$, and denote
this projection by $p_j=\pi_{\mathcal T}(w_j)$.

By the properties of $W$, this projection is orthogonal with respect to the standard Riemannian metric on $\mathbb R^N$: for each vector $w_j\in\mathcal T$ we have $(w_j-\pi_{\mathcal T}(w_j))\bot\pi_{\mathcal T}(\mathcal T)$.
Indeed, expand
$$w_j=p_j+q_j=\pi_{\mathcal T}(w_j)+q_j
$$
and verify that
$q_j$ is orthogonal to $\pi_{\mathcal T}(\mathcal T)$. Take $u\in\pi_{\mathcal T}(\mathcal T)$. First,
suppose that $u$ is a ``basis'' vector in this image:
$u$ is one of the row-vectors of the block ${\cal V}_k$ for some $1\leq k\leq M$. If $p_j\in{\cal
V}_l$ and $l>k$ then $q_j\bot u$ by choice of $W$. If
$l\leq k$ then obviously $q_j\bot u$ since $q_j$ and $u$ have no
nonzero coordinates with the same indices. Since $j$ is an
arbitrary number and $u$ is an arbitrary basis vector, it follows
that all $q_j$ are orthogonal to $\pi_{\mathcal T}(\mathcal T)$. Consequently,
$\operatorname{span}\{q_j, j=1,\ldots,N-\widetilde{N}\}\bot\pi_{\mathcal T}(\mathcal T)$.

Denote by $\pi_j$ the vector obtained from $w_j$ by putting
$$
\begin{cases}
(\pi_j)_{l(j)}&=(w_j)_{l(j)},\\
(\pi_j)_k&=0\text{ for }k\neq l(j).
\end{cases}
$$
Put $\operatorname{deg}\pi_j=k(j)$. By construction, the sum of the degrees of
$\pi_1,\ldots, \pi_{N-\widetilde{N}}$ equals the sum of degrees of the
tangent vectors $(\theta_x^{-1})_*\langle w_1\rangle,\ldots,(\theta_x^{-1})_*\langle w_{N-\widetilde{N}}\rangle$.

{\bf Step III.} Consider the vectors
\begin{equation}\label{bsigma}
\{\sigma_1,\ldots,\sigma_{\widetilde{N}}\}=\{e_1,\ldots,e_{N}\}\setminus\Bigl\{\frac{\pi_1}{|\pi_1|},\ldots,
\frac{\pi_{N-\widetilde{N}}}{|\pi_{N-\widetilde{N}}|}\Bigr\}
\end{equation}
 of the standard basis in $\mathbb R^N$.
Since $\theta_x$ is a diffeomorphism, we have
$$
\operatorname{rank}(([(\theta_x)_*(0)\langle\sigma_i\rangle]\varphi)(x))_{i=1}^{\widetilde{N}}=\widetilde{N}\iff\operatorname{rank}([\sigma_i\psi](0))_{i=1}^{\widetilde{N}}=\widetilde{N}.
$$
Here the symbol $(\theta_x)_*(0)\langle\sigma_i\rangle=D\theta_x(0)\langle\sigma_i\rangle$ stands for the action of the differential of $\theta_x$ at $0$ on the vector $\sigma_i$, $i=1,\ldots,\widetilde{N}$.

Verify that the rank of $([\sigma_i\psi](0))_{i=1}^{\widetilde{N}}$ equals
$\widetilde{N}$. Indeed, assume the contrary. Then $\ker D\psi(0)\cap\operatorname{span}\{\sigma_1,\ldots,\sigma_{\widetilde{N}}\}\neq0$, and therefore
$$
\dim \mathcal T\cap\operatorname{span}\{\sigma_1,\ldots,\sigma_{\widetilde{N}}\}\geq 1.
$$
Consider the coordinates of $w\in
\mathcal T\cap\operatorname{span}\{\sigma_1,\ldots,\sigma_{\widetilde{N}}\}$. On the one hand, since
$w\in \mathcal T$, the choice of the matrix $W$ implies that we can
expand $w$ as
$$
w=\sum\limits_{i=1}^{N-\widetilde{N}}a_iw_i.
$$
Put $i_0=\max\{i:a_i\neq0\}$. Since $w_i=(w_i-\pi_i)+\pi_i$, it follows that
$$
w=\sum\limits_{i=1}^{N-\widetilde{N}}a_i(w_i-\pi_i)+a_i\pi_i
=\Bigl(\sum\limits_{i=1}^{i_0-1}a_i[(w_i-\pi_i)+\pi_i]+a_{i_0}(w_{i_0}-\pi_{i_0})\Bigr)+a_{i_0}\pi_{i_0}.
$$
Recall that by the choice of $\{\pi_i\}_{i=1}^{N-\widetilde{N}}$, the only nonzero coordinate of $\pi_{i_0}$ has index$l(i_0)$, and the nonzero coordinates of
$$\Bigl(\sum\limits_{i=1}^{i_0-1}a_i[(w_i-\pi_i)+\pi_i]+a_{i_0}(w_{i_0}-\pi_{i_0})\Bigr)$$
have indices strictly less than~$l(i_0)$.
On the other hand, since
$w\in\operatorname{span}\{\sigma_1,\ldots,\sigma_{\widetilde{N}}\}$ it follows from
{\eqref{bsigma}} that all coefficients of
$\pi_1,\ldots, \pi_{N-\widetilde{N}}$ in this expansion must vanish. Consequently, $a_{i_0}=0$. Thus,
we arrive at a contradiction.

{\bf Step IV.} Verify that the sum of degrees of basis vector fields $\{X_{i_j}\}_{j=1}^{\widetilde{N}}$ such that $X_{i_j}(x)=(\theta_x)_*(0)\langle\sigma_j\rangle$, $j=1,\ldots,\widetilde{N}$, equals $\nu_0(x)$.

{\bf (i)} Assume on the contrary that there exist standard vectors
$\delta_1,\ldots,\delta_{\widetilde{N}}$ with
\begin{equation}\label{nu0}
\nu_0(x)=\sum\limits_{i=1}^{\widetilde{N}}\operatorname{deg}\delta_i<\sum\limits_{i=1}^{\widetilde{N}}\operatorname{deg}\sigma_i
\end{equation}
and $\operatorname{rank}(\delta_i\psi)_{i=1}^{\widetilde{N}}=\widetilde{N}$ (here $\operatorname{deg}\delta_i$ stands for the value $\operatorname{deg}[(\theta_x)_*(0)\langle\delta_i\rangle]$, $i=1,\ldots,\widetilde{N}$). Since
$\{\delta_1,\ldots,\delta_{\widetilde{N}}\}\neq\{\sigma_1,\ldots,\sigma_{\widetilde{N}}\}$, there exists at least one $\frac{\pi_j}{|\pi_j|}=\delta_k$
for some $j$ and $k$. Consequently, there exists a vector $v$ satisfying $v+\frac{\pi_j}{|\pi_j|}\in \mathcal T$. Moreover, the index of the last nonzero
component of $v$ is strictly less than $l(j)$ (see the matrix $W$ of {\eqref{matw}}).

{\bf (ii)} Consider the chosen vector $v$. It is evident that
$v\notin \mathcal T$ (otherwise, $\pi_j\in \mathcal T$ and $\delta_k\in \mathcal T$, which is impossible). Observe that $\operatorname{rank} ({\tau_i}\psi)_{i=1}^{\widetilde{N}}<\widetilde{N}$, where
$\tau_k=\frac{1}{2}\bigl(v+ \frac{\pi_j}{|\pi_j|}\bigr)\in \mathcal T$, and $\tau_i=\delta_i$ for all $i\neq k$.
Consequently, $\operatorname{rank} ({\lambda_i}\psi)_{i=1}^{\widetilde{N}}=\widetilde{N}$, where
$\lambda_k=v$ and $\lambda_i=\delta_i$ for all $i\neq k$. Indeed,
assume on the contrary that $\operatorname{rank}
({\lambda_i}\psi)_{i=1}^{\widetilde{N}}<\widetilde{N}$. Then
\begin{equation}\label{matrices}
(\delta_i\psi)_{i=1}^{\widetilde{N}}=2(\tau_i\psi)_{i=1}^{\widetilde{N}}-(\lambda_i\psi)_{i=1}^{\widetilde{N}}.
\end{equation}
The assumption implies that
$v\psi=\sum\limits_{i=1}^{\widetilde{N}}a_i[\delta_i\psi]$ because $\operatorname{rank}(\delta_i\psi)_{i=1}^{\widetilde{N}}=\widetilde{N}$. Since $2\tau_k\psi=\bigl(v+ \frac{\pi_j}{|\pi_j|}\bigr)\psi=0$, it
follows from {\eqref{matrices}} that the column $k$ of $(\delta_i\psi)_{i=1}^{\widetilde{N}}$ is
equal to $-v\psi=-\sum\limits_{i=1}^{\widetilde{N}}a_i[\delta_i\psi]$, and,
consequently, $\operatorname{rank}(\delta_i\psi)_{i=1}^{\widetilde{N}}<\widetilde{N}$. Thus, we arrive at a contradiction.

Since $\operatorname{rank}(\lambda_i\psi)_{i=1}^{\widetilde{N}}=\widetilde{N}$, assuming the
contrary and applying similar arguments, we
infer that there exists at least one coordinate $(v)_j$ of $v$
such that $\operatorname{rank}(\mu_i\psi)_{i=1}^{\widetilde{N}}=\widetilde{N}$, where
$\mu_k=e_j$ and $\mu_i=\delta_i$ for all $i\neq k$.

{\bf (iii)} Put $$
l=\min\{j:\operatorname{rank}(\mu_i\psi)_{i=1}^{\widetilde{N}}=\widetilde{N},\text{ where
}\mu_k=e_j,\ \mu_i=\delta_i\text{ for }i\neq k\}.
$$ Note that
$l$ is strictly less than $l(j)$ (the index of the only
nonzero coordinate of $\pi_j$). If $\operatorname{deg}e_l<\operatorname{deg}\pi_j$ then
we obtain a contradiction with the assumption that
$\sum\limits_{i=1}^{\widetilde{N}}\operatorname{deg}\delta_i=\nu_0(x)$. If
$\operatorname{deg}e_l=\operatorname{deg}\pi_j$ then we consider the collection
$\{\mu_i\}_{i=1}^{\widetilde{N}}$ instead of $\{{\delta}_i\}_{i=1}^{\widetilde{N}}$,
and repeat the previous arguments.

{\bf (iv)} Observe that after finitely many iterations of the arguments described in substeps {\bf (ii)--(iii)}, we obtain a collection satisfying
$$
\sum\limits_{i=1}^{\widetilde{N}}\operatorname{deg}\mu_i<\sum\limits_{i=1}^{\widetilde{N}}\operatorname{deg}\delta_i=\nu_0(x)
$$
(here $\operatorname{deg}\mu_i$ stands for the value $\operatorname{deg}[(\theta_x)_*(0)\langle\mu_i\rangle]$, $i=1,\ldots,\widetilde{N}$). This contradicts to {\eqref{nu0}}.

{\bf Step V.} {\bf (i)} Denote by $\mathcal S$ the plane
$\operatorname{span}\{p_1,\ldots,p_{N-\widetilde{N}}\}$, where the vectors
$p_1,\ldots,p_{N-\widetilde{N}}$ are defined in Step {\bf II}, and
consider $\mathcal S\cap\operatorname{Box}_2(0,r)$. Without
loss of generality we may assume that $p_1,\ldots,p_{N-\widetilde{N}}$ are orthonormal (it suffices to carry out the procedure described in Step~{\bf II} within each $\mathcal V_i$, $i=1,\ldots, M$).
Verify that this intersection equals a direct product of balls
of radii $r^k$, $k=1,\ldots, M$, in some basis.

To this end, we consider a new orthonormal basis $\mathcal Oe_1,\ldots, \mathcal Oe_N$ in $\mathbb
R^{N}$, where the orthogonal
transformation $\mathcal O$ satisfies $\mathcal O(\operatorname{span}\{\mathcal V_i\})=\operatorname{span}\{\mathcal V_i\}$, $i=1,\ldots, M$, and the image of each $p_j$ is the standard basis
vector $e_{i_j}$, $j=1,\ldots,N-\widetilde{N}$. Then the $d_2$-ball in the initial basis coincides
with the one in the new metric: ${\cal
O}(\operatorname{Box}_2(0,r))=\operatorname{Box}_2(0,r)$. Thus, we have constructed a
basis for $\mathcal S$ consisting of $N-\widetilde{N}$ standard basis vectors.

Next, ${\cal O}(\operatorname{Box}_2(0,r)\cap \mathcal S)=\operatorname{Box}_2(0,r)\cap{\cal
O}(\mathcal S)=\operatorname{Box}_2^{{\cal O}(\mathcal S)}(0,r)$, where the latter
is a ball in the metric ${d_2|}_{{\cal O}(\mathcal S)}$. By the choice of ${\cal O}$, it equals
a direct product of Euclidean balls.

{\bf (ii)} Obviously, the $(N-\widetilde{N})$-dimensional Lebesgue measure
of $\operatorname{Box}_2^{{\cal O}(\mathcal S)}(0,r)$ equals $Cr^{\nu-\nu_0(x)}$, where
$C$ depends only on $x$. Since
${\cal O}$ is an orthogonal transformation, so is the inverse mapping ${\cal
O}^{-1}$, and the $(N-\widetilde{N})$-dimensional
Lebesgue measure of $\operatorname{Box}_2(0,r)\cap \mathcal S={\cal O}^{-1}(\operatorname{Box}_2^{{\cal
O}(\mathcal S)}(0,r))$ is the same; thus, it also equals
$Cr^{\nu-\nu_0(x)}$.

Observe that at regular points the $(N-\widetilde{N})$-dimensional Lebesgue
measure of $\mathcal S\cap\operatorname{Box}_2(0,r)$ equals
$\prod\limits_{k=1}^{M}\omega_{n_k-\tilde{n}_k}r^{\nu-\tilde{\nu}}$.

{\bf Step VI.} On this step we show that there exists $r_0>0$ such that for $r\leq r_0$ the set
$$
\pi_{\mathcal T}(\mathcal T\cap\operatorname{Box}_2(0,r))
$$
coincides with the $o(r)$-neighborhood of $\mathcal S\cap\operatorname{Box}_2(0,r)$ in
$\mathcal S$, where $o(r)$ is taken with respect to the metric $d_2$. From
this, we deduce that the ${\cal H}^{N-\widetilde{N}}$-measure
distortion of $\pi_{\mathcal T}$ equals
$$
\lim\limits_{r\to0}\frac{{\cal H}^{N-\widetilde{N}}(\pi_{\mathcal T}(\mathcal T\cap
\operatorname{Box}_2(0,r)))}{{\cal
H}^{N-\widetilde{N}}(\mathcal T\cap\operatorname{Box}_2(0,r))}=\lim\limits_{r\to0}\frac{{\cal
H}^{N-\widetilde{N}}(\mathcal S\cap\operatorname{Box}_2(0,r))}{{\cal
H}^{N-\widetilde{N}}(\mathcal T\cap\operatorname{Box}_2(0,r))}.
$$

{\bf (i)} Fix a vector $w_j$. Show that there exists ${r_0}_j>0$ such that the length of $\mathbb Rw_j\cap\operatorname{Box}_2(0,r)$ equals $O(r^{k(j)})$ for
all $r\in (0,{r_0}_j)$. Indeed, fix ${r_0}_j>0$ and a take point
$\bar{u}\in\mathbb R_+w_j\cap\partial\operatorname{Box}_2(x,{r_0}_j)$ and the corresponding vector $u$ whose coordinates coincide with those of  $\bar{u}$. Here
$$
\mathbb R_+w_j=\{z\in\mathbb R^N: z=aw_j, a\in\mathbb R_+\}.
$$
Then
$u=\sum\limits_{k=1}^{k(j)}u_k$, where $$ u_k=u_k({r_0}_j)\in
\operatorname{span}\{\mathcal V_k\},
$$ and
$d_2(0,\bar{u})={r_0}_j$. Consequently, $|u_k({r_0}_j)|\leq ({r_0}_j)^k$; here
$|\cdot|$ denotes the Euclidean length.

We verify that $|u_k(r)|=O(r^{k(j)})$ for all $r\in(0,{r_0}_j)$, $k=1,\ldots, k(j)$. The vector $u(r)=\sum\limits_{k=1}^{k(j)}u_k(r)$ has the same coordinates as the point $\mathbb R_+w_j\cap\partial\operatorname{Box}_2(0,r)$. Assume on the contrary that for every ${r_0}_j$ and for every $K<\infty$ there exist $r< {r_0}_j$ and $u_l(r)$ with
$$
|u_l(r)|>Kr^{k(j)}.
$$
Fix such $r_0, K$ and $r<r_0$, and put $\alpha_r=\frac{|u|}{|u(r)|}$. Then we can represent $u(r)$ as
$$u(r)=\alpha_ru=\sum\limits_{k=1}^{k(j)}\alpha_ru_k=\sum\limits_{k=1}^{k(j)}u_k(r).
$$
On the one hand,
$Kr^{k(j)}<|u_l(r)|=|\alpha_ru_l|=|\alpha_r||u_l|$. Therefore,
$$|\alpha_r|>K\frac{r^{k(j)}}{|u_l|}.$$

On the other hand, $k=k(j)$, in view of its definition, satisfies
$|u_{k(j)}(r)|\leq Lr^{k(j)}$ for all $r\in(0,r_0]$. Since
$u_{k(j)}(r)=\alpha_r\cdot u_{k(j)}$, we have
$$
Lr^{k(j)}\geq|u_{k(j)}(r)|=|\alpha_r|\cdot
|u_{k(j)}|>K{r^{k(j)}}\frac{|u_{k(j)}|}{|u_l|}.
$$
Observe that $\frac{|u_{k(j)}|}{|u_l|}$ depends only on $w_j$, and is independent of $r>0$. Obviously, the equality
$$
Lr^{k(j)}>Kr^{k(j)}
$$
is violated for $K>L$. Thus we arrive at a contradiction. Consequently, $|u_k|=O(r^{k(j)})$,
$r\in(0,r_0)$, for all $k=1,\ldots,k(j)$.

{\bf (ii)} Consider $\mathcal T\cap\operatorname{Box}_2(0,r)$, and the
projection $\pi_{\mathcal T}:\mathcal T\to \mathcal S$. By the properties of the matrix
$W$ (see {\eqref{matw}}), it is bijective. Verify that
\begin{equation}\label{distboundary}
d_2(0,\partial[\pi_{\mathcal T}(\mathcal T\cap\operatorname{Box}_2(0,r))])=r(1+o(1)),
\end{equation}
where $\partial$ stands for the boundary relative to the plane $\mathcal T$.
First, consider the basis elements $w$ in $\mathcal T$. Represent each of them as
$w=y+z$, where $\pi_{\mathcal T}(y+z)=y$. The properties of $\pi_{\mathcal T}$ yield
$\operatorname{deg} z<\operatorname{deg} y$, where the degree of the vector
$u=\sum\limits_{i=1}^{\widetilde{N}}u_ie_i$ is defined as
$$
\operatorname{deg}
u=\operatorname{deg}\Bigl(\sum\limits_{i=1}^{\widetilde{N}}u_iX_i\Bigr)=\max\limits_{i=1,\ldots,N}\{\operatorname{deg}
X_i:u_i\neq0\}.
$$
If $w\in \mathcal T$ is not a basis vector, then we expand it in the basis
$\{w_i\}_{i=1}^{N-\widetilde{N}}$ as
$$
w=\sum\limits_{k=1}^{N-\widetilde{N}}a_kw_k=\sum\limits_{k=1}^{N-\widetilde{N}}a_ky_k+\sum\limits_{k=1}^{N-\widetilde{N}}a_kz_k,
$$
where $\pi_{\mathcal T}(w_k)=y_k$, $k=1,\ldots,N-\widetilde{N}$. Next, it
suffices to note that
$$
\operatorname{deg}\Bigl(\sum\limits_{k=1}^{N-\widetilde{N}}a_ky_k\Bigr)>\operatorname{deg}\Bigl(\sum\limits_{k=1}^{N-\widetilde{N}}a_kz_k\Bigr).
$$

Thus, we can represent every $w\in \mathcal T$ as $w=y+z$, where
$\pi_{\mathcal T}(w)=y$ and $\operatorname{deg} y>\operatorname{deg} z$.

Define a new quasimetric ${d_2^0}_E$ on $\operatorname{Box}_2(0,r)$. For
$v,w\in\operatorname{Box}_2(0,r)$ put ${d_2^0}_E(v,w)=d_2(0,w-v)$, where $w-v$
denotes the Euclidean difference. This definition implies that
$\operatorname{Box}_{2}(0,r)$ coincides with the ball $\operatorname{Box}_{2E}(0,r)$ of radius $r$ centered at
$0$ in the metric ${d_2^0}_E$. Consequently, for proving {\eqref{distboundary}}, it suffices to show that
$$
{d_2^0}_E(0,\partial[\pi_{\mathcal T}(\mathcal T\cap\operatorname{Box}_2(0,r))])=r(1+o(1)),
$$
where $\partial$ stands for the boundary relative to the plane $\pi_{\mathcal T}(\mathcal T)$.
Take $w=y+z$ with $d_2(0,w)=r$ and $\operatorname{deg} y>\operatorname{deg} z$. Then
$$
d_2(0,y+z)={d_2^0}_E(0,y+z)=r,
$$
$\rho(0,y+z)=O(r^{\operatorname{deg} y})$ (see Step {\bf V}, substep {\bf (i)}),
and $\rho(0,y)=O(r^{\operatorname{deg} y})$ since $\pi_{\mathcal T}$ is a linear mapping.
Then, we have $\rho(y,y+z)=O(r^{\operatorname{deg} y})$ because $y\bot z$.

Observe that ${d_2^0}_E(y,y+z)=d_2(0,z)$. Represent $z$ as $z=\sum\limits_{i=1}^{I(z)}z_i$, where
$z_i\in\mathcal V_i$.
Here $\operatorname{deg} z_{I(z)}=I(z)<\operatorname{deg} y$ because $\operatorname{deg} z<\operatorname{deg} y$. Since
$$
O(r^{\operatorname{deg}
y})=\rho(y,y+z)\sim\max\limits_{\substack{i=1,\ldots,I(z)\\j=1,\ldots,\dim
H_i/H_{i-1}}}\{|(z_i)_j|\},
$$
it follows that $|z_i|\leq O(r^{\operatorname{deg} y})$.

This implies that
\begin{multline*}
d_2(0,z)=\max\limits_{i=1,\ldots,I(z)}\{|z_i|^{\frac{1}{i}}\}\leq
\max\limits_{i=1,\ldots,I(z)}\{|z_i|^{\frac{1}{I(z)}}\}\\
\leq O(r^{\frac{\operatorname{deg} y}{I(z)}})\leq O(r^{1+\frac{1}{M-1}})=r\cdot
o(1),
\end{multline*}
where $o(1)$ is at most ${\cal C}r^{\frac{1}{M-1}}$ for
some $0<{\cal C}<\infty$ independent of the point of $\operatorname{Box}_2(0,r)$.

Proposition {\ref{prop21vg}} implies that
${d_2^0}_E(v,w)\leq{d_2^0}_E(v,u)+c{d_2^0}_E(u,w)$ for
$u,v,w\in\operatorname{Box}_2(0,r)$ and some $c>0$ (it suffices to put
$\widehat{F}^i_{\mu,\beta}(x)\equiv0$ in Proposition {\ref{prop21vg}}).
Consequently,
$$d_2(0,y)={d_2^0}_E(0,y)\leq{d_2^0}_E(0,z+y)+c{d_2^0}_E(y,z+y)=r(1+o(1)),
$$ and
{\eqref{distboundary}} follows.

Thus, the measure of $\mathcal T\cap\operatorname{Box}_2(x,r)$ is equivalent to
$Cr^{\nu-\nu_0(x)}$ as $r\to0$. The theorem follows.
\end{proof}

\begin{defn} Let $\xi:\mathbb M\to\widetilde{\mathbb M}$ be a mapping of two Carnot
manifolds. Fix $x\in\mathbb M$. The $d_2$-{\it distortion of $\xi$ at $y$ with
respect to} $x$ equals
$$
\frac{d_2(\xi(y),\xi(x))}{d_2(y,x)},
$$
and the $\rho$-distortion equals
$$
\frac{\rho(\xi(y),\xi(x))}{\rho(y,x)}.
$$
\end{defn}

\medskip
To simplify notation, we denote the Gram determinant $\sqrt{\det(AA^*)}$
for a matrix~$A$ by $\mathcal D(A)$. We denote the Gram determinant $\sqrt{\det(B^*B)}$ for a matrix~$B$ by $\widetilde{\mathcal D}(B)$.

\begin{prop} The matrix of the differential of the mapping $\theta_x$, $x\in\mathbb M$,
at $0$ equals the identity matrix.
\end{prop}

Henceforth we denote the Riemann tensor at $y$ by $g(y)$.

\begin{thm}\label{Leb_meas} Suppose that $x\in\varphi^{-1}(t)$ is a
regular point. Then

{\bf(I)} In a neighborhood of
$0=\theta_x^{-1}(x)$ there exists a mapping from
$T_0[\psi^{-1}(t)]\cap\operatorname{Box}_2(0,r(1+o(1)))$ to
$\psi^{-1}(t)\cap\operatorname{Box}_2(0,r)$ such that both $d_2$- and
$\rho$-distortions with respect to $0$ are equal to $1+o(1)$ at every $y\in T_0[\psi^{-1}(t)]\cap\operatorname{Box}_2(0,r(1+o(1)))$, where
$o(1)$ is uniform in $x=\theta_x(0)\in\mathcal U\Subset\mathbb M$ and in $y\in T_0[\psi^{-1}(t)]\cap\operatorname{Box}_2(0,r(1+o(1)))$.

{\bf (II)} The ${\cal H}^{N-\widetilde{N}}$-measure of
$\varphi^{-1}(t)\cap\operatorname{Box}_2(x,r)$ equals
$$
\widetilde{\mathcal D}(g|_{\ker
D\varphi(x)})\cdot\prod\limits_{k=1}^{M}\omega_{n_k-\tilde{n}_k}\cdot
\frac{\mathcal D({D}\varphi(x))}{\mathcal D(\widehat{D}\varphi(x))}r^{\nu-\tilde{\nu}}(1+o(1))
$$
with $o(1)\to0$ as $r\to0$, where $o(1)\to0$ uniformly in $x\in\mathcal U\Subset\mathbb M$.
\end{thm}

In the proof of Theorem {\ref{Leb_meas}} we use the following
notation introduced in Theorem {\ref{tang_meas}}. We denote the
mapping $\varphi\circ\theta_x$ by $\psi$, and the tangent space
$T_0[\psi^{-1}(t)]$ by $\mathcal T$. We also use the mapping $\pi_{\mathcal T}$
defined in Step {\bf III} and the image $\mathcal S=\pi_{\mathcal T}(\mathcal T)$.
\medskip

\begin{proof}[Proof of Theorem {\ref{Leb_meas}}] Notice that without loss of generality we may assume
that $\widehat{D}\psi(z)$ is strictly
separated from $0$ on $(\ker\widehat{D}\psi(0))^{\bot}$ for
$z\in\operatorname{Box}_2(0,r)$, i.~e., $\widehat{D}\psi(z)(v)\geq\alpha>0$ for all
$z\in\operatorname{Box}_2(0,r)$ and $v\in(\ker\widehat{D}\psi(0))^{\bot}$ with $|v|=1$ (here $|\cdot|$ denotes the Euclidean norm).

{\bf Step I.}  Verify that $C$ in {\eqref{constc}} at a regular point equals
$\frac{\mathcal D({D}\varphi(x))}{\mathcal D(\widehat{D}\varphi(x))}$ up to a Riemannian
factor. In particular, we show that
$\frac{\mathcal D({D}\varphi(x))}{\mathcal D(\widehat{D}\varphi(x))}$ equals the measure
distortion under the mapping $\pi_{\mathcal T}$ defined
in Step {\bf II} of Theorem {\ref{tang_meas}}.

To this end, consider the projection of the normal space $\mathcal N$ to
$\mathcal T$ onto $H_{\widetilde{M}}(0)$ constructed by the choice of vectors
written as rows in the matrix of $\widehat{D}\psi(0)$ (recall that the matrix of $D\psi(0)$ equals the matrix of $([(\theta_x^{-1})_*X_1]\psi,\ldots,[(\theta_x^{-1})_*X_N]\psi)$, and in the bases $\{e_i\}_{i=1}^N$ and $\{e_i\}_{i=1}^{\widetilde{N}}$ the matrices of ${D}\psi(0)$ and $\widehat{D}\psi(0)$ have the structure similar to {\eqref{riemdiff}} and {\eqref{subriemdiff}}, respectively). Denote this
projection by $\pi_{\mathcal N}$ and its image by
$\mathcal L=\pi_{\mathcal N}(\mathcal N)$. Verify that $\mathcal S$ is orthogonal to $\mathcal L$. Write a basis
for $\mathcal N$ as the row vectors of the matrix $D\psi(0)$ in the basis
$\{e_i\}_{i=1}^{N}$.

{\bf (i)} Without loss of generality we may assume that $\pi_{\mathcal N}$ is an orthogonal projection (see Theorem
{\ref{tang_meas}}). Indeed, the matrix of $D\psi(0)$ admits ``gradings'' $\widetilde{\cal B}_1,\ldots,\widetilde{\cal B}_{\widetilde{M}}$ of the columns
and $\widetilde{\cal A}_1,\ldots,\widetilde{\cal A}_{\widetilde{M}}$ of the rows similar to those of the matrix $\Delta$ of Step {\bf II} substep {\bf (ii)} of
Theorem {\ref{tang_meas}}. Namely, we say that a column vector belongs to the block $\widetilde{\mathcal B}_k$ if its index is at least $\dim H_{k-1}+1$ (here we assume that $\dim H_0=0$) and at most $\dim H_k$, $k=1,\ldots,\widetilde{M}$, and we say that a row vector belongs to the block $\widetilde{\mathcal A}_l$ if the index of its first nonzero element is at least $\dim H_{k-1}+1$ (here we assume that $\dim H_0=0$) and at most $\dim H_k$, $k=1,\ldots,\widetilde{M}$. Put $\widetilde{\cal V}_k=\widetilde{\cal B}_k\cap \widetilde{\cal
A}_k$, $k=1,\ldots,\widetilde{M}$.

Next, we transform the blocks $\widetilde{\mathcal B}_k\setminus\widetilde{\mathcal V}_k$,
$k=1,\ldots,\widetilde{M}$.  Define the required transformation of the matrix of
$D\psi(0)$ by induction. For $k=1$, we have nothing to transform;
thus, the base of induction holds. Suppose that  we have
transformed the blocks $\widetilde{\mathcal B}_k\setminus\widetilde{\mathcal V}_k$,
$k=1,\ldots,l$, $l\geq 1$. Assume that $\widetilde{\mathcal
V}_{l+1}\neq\varnothing$ (otherwise, the transformation is trivial).
Replace the blocks $\widetilde{\cal A}_{1},\ldots,\widetilde{\cal A}_{l}$ by the
projections of their row vectors onto the space
$$
(\operatorname{span}\{\mathcal A_{l+1}\})^{\bot}\cap\operatorname{span}\{\widetilde{\cal A}_1,\ldots,\widetilde{\cal
A}_{l+1}\}.
$$
This projection leaves the blocks
$\widetilde{\mathcal B}_k\setminus\widetilde{\mathcal V}_k$, $k=1,\ldots,l$, $l\geq 1$,
unchanged because of the ``triangular'' structure of the
matrix of $D\psi(0)$.

Moreover, the rows of the block $\widetilde{\mathcal B}_{k}\setminus{\cal V}_{k}$
are orthogonal to those of $\widetilde{\mathcal V}_k$, $k=1,\ldots,\widetilde{M}$.
It results
\begin{equation*}
\left(%
\begin{array}{ccccccccccccc}
 {\cal V}_1 & {\cal V}_2^{\bot} &  & {\cal V}_3^{\bot} &  & * & * &  & \vdots & \vdots & * & \ldots & * \\
0  &{\cal V}_2  &   &  &  & \vdots & \vdots & * &  & * & * & \ldots & * \\
0 & 0 &  & {\cal V}_3 &  &  &  & * & {\cal V}_{\widetilde{M}-1}^{\bot} & {\cal V}_{M}^{\bot} & * & \ldots & * \\
  \vdots  & \vdots & 0 & 0 &  & \ddots & \ddots & \ddots & * & & * & \ldots & * \\
0 & 0 &  &  & \vdots & 0 & 0 & 0 & {\cal V}_{\widetilde{M}-1} & & * & \ldots & * \\
0 & 0 & 0 & 0 & 0 & 0 & 0 & 0 & 0 &{\cal V}_{\widetilde{M}} & * & \ldots & * \\
\end{array}%
\right)
\end{equation*}

By the construction, the blocks ${\cal
V}_1,\ldots,{\cal V}_{\widetilde{M}}$ constitute the matrix of
$\widehat{D}\psi(0)$ (see {\eqref{subriemdiff}}), and $\dim \mathcal L=\widetilde{N}$ at regular points.

{\bf (ii)} It is easy to see that $\mathcal L$ is orthogonal to $\mathcal S$.

Indeed, take two vectors $v\in \mathcal L$ and~$w\in \mathcal S$. First, suppose that
$v$ and~$w$ are the images of arbitrary basis vectors in $\mathcal N$ and $\mathcal T$
respectively. Then $v\in\widetilde{\mathcal V}_i$ and $w\in{\mathcal V}_j$ for some $i$ and $j$.
If $i\neq j$ then obviously $v\bot w$ (since they have no corresponding nonzero components).

Suppose that $i=j$ and consider the preimages of $v$ and $w$. By
construction, they are orthogonal. Note that the preimage of $v$
belongs to $\mathcal N$, and we can write it as $(0,\ldots,0,\tilde{v},*)$, where
the nonzero part begins with some element $\tilde{v}$ of $v$. The preimage of
$w$ belongs to $\mathcal T$, and it can be written as $(*,\tilde{w},0,\ldots,0)$,
where the last component of the nonzero part is some element $\tilde{w}$ of
$w$. Consequently,
$$
0=(0,\ldots,0,\tilde{v},*)\cdot(*,\tilde{w},0,\ldots,0)^T=(v,w),
$$
and $v\bot w$.

Since $v$ and $w$ are the images of arbitrary basis
vectors in $\mathcal T$ and $\mathcal N$, we have $\mathcal S\bot \mathcal L$.

{\bf (iii)} Thus, $\mathcal S$ is orthogonal to $\mathcal L$. Moreover, the
projection of $\mathcal T$ onto $\mathcal S$ is also orthogonal (see Theorem
{\ref{tang_meas}}): $(v-\pi_{\mathcal T}(v))\bot \pi_{\mathcal T}(v)$ for
all $v\in \mathcal T$. Let us verify that the measure distortions
under the projections $\pi_{\mathcal T}$ and $\pi_{\mathcal N}$ coincide.

{\bf Step II.} {\bf (i)} Show that
$$
\dim \mathcal T-\dim (\mathcal T\cap \mathcal S)=\dim \mathcal N-\dim (\mathcal N\cap \mathcal L).
$$
Put $l=\dim (\mathcal T\cap \mathcal S)$. On the one hand, since $\mathcal N\bot \mathcal T$ and
$\mathcal L\bot \mathcal S$, it follows that $\operatorname{span}\{\mathcal N,\mathcal L\}\bot (\mathcal T\cap \mathcal S)$, and so
$\operatorname{span}\{\mathcal N,\mathcal L\}\subset(\mathcal T\cap \mathcal S)^{\bot}$. On the other hand, if
$v\bot(\mathcal T\cap \mathcal S)$ then $v\in\operatorname{span}\{\mathcal N,\mathcal L\}$. Indeed, assume on the
contrary that $v\notin\operatorname{span}\{\mathcal N,\mathcal L\}$. Then $v=v_{\mathcal N\mathcal L}+v_{(\mathcal N\mathcal L)^{\bot}}$,
where $v_{\mathcal N\mathcal L}\in\operatorname{span}\{\mathcal N,\mathcal L\}$ and
$v_{(\mathcal N\mathcal L)^{\bot}}\in(\operatorname{span}\{\mathcal N,\mathcal L\})^{\bot}$ with $v_{(\mathcal N\mathcal L)^{\bot}}\neq0$. Consequently, since $v_{(\mathcal N\mathcal L)^{\bot}}\bot\mathcal N$ and $v_{(\mathcal N\mathcal L)^{\bot}}\bot\mathcal L$, we obtain
$v_{(\mathcal N\mathcal L)^{\bot}}\in(\mathcal T\cap \mathcal S)$ and arrive at a contradiction.

Thus, $\operatorname{span}\{\mathcal N,\mathcal L\}=(\mathcal T\cap \mathcal S)^{\bot}$, and
$\dim\operatorname{span}\{\mathcal N,\mathcal L\}=N-l$. Since $\dim\mathcal N=\dim\mathcal L$, we get
$$\dim (\mathcal N\cap
\mathcal L)=2\dim \mathcal N-\dim\operatorname{span}\{\mathcal N,\mathcal L\}=2\widetilde{N}-N+l.
$$ Consequently,
\begin{equation}\label{dims}
\dim \mathcal T-\dim (\mathcal T\cap \mathcal S)=N-\widetilde{N}-l=\dim \mathcal N-\dim (\mathcal N\cap \mathcal L).
\end{equation}

{\bf (ii)} Recall that $\mathcal L=\ker\widehat{D}\psi(0)$ at regular points (see Step {\bf I}, substep {\bf (ii)}). Verify that $\dim(\mathcal T\cap \mathcal L)=\dim(\mathcal N\cap \mathcal S)=0$. Indeed, assume on the
contrary that there exists $h\in \mathcal T\cap \mathcal L$, $h\neq0$. Then, $h\bot \mathcal N$ and $h\bot
\mathcal S$. Since $(\mathcal T\cap \mathcal L)^{\bot}=\operatorname{span}\{\mathcal N, \mathcal S\}$, which we can justify similarly to the substep {\bf (i)} above, it follows that $\dim\operatorname{span}\{\mathcal N,
\mathcal S\}\leq N-1$. Consequently, from $\dim\operatorname{span}\{\mathcal N\}=\widetilde{N}$ and $\dim\operatorname{span}\{\mathcal S\}=N-\widetilde{N}$, we infer that $\dim (\mathcal N\cap \mathcal S)\geq 1$ and there
exists $h^{\bot}\in(\mathcal N\cap \mathcal S)$; thus, $\widehat{D}\psi(0)$ vanishes
on some vector in~$\mathcal N$. This implies that one of $\widetilde{N}$ vectors in $\mathcal N$ is orthogonal to $\mathcal L$. Consequently, its image under
$\pi_{\mathcal N}$ vanishes since every vector in $\mathcal N$ can be
uniquely represented as a sum of a vector in $\mathcal L$ and a vector
in $\mathcal L^{\bot}$. Thus, $\pi_{\mathcal N}$
has a nonzero kernel, and $\dim (\pi_{\mathcal N}(\mathcal N))<\widetilde{N}$. This implies that
$\dim[\pi_{\mathcal N}(\mathcal N)]=\dim[(\ker\widehat{D}\psi(0))^{\bot}]<\widetilde{N}$, which
contradicts the regularity of $x$.

Hence, $\dim(\mathcal T\cap \mathcal L)=0$ and $\dim (\mathcal N\cap \mathcal S)=0$.

{\bf (iii)} Put $q=\dim \mathcal T-\dim (\mathcal T\cap \mathcal S)$.
Consider the maximal subset $\mathcal S^{\prime}$ of $\mathcal S$ satisfying $\mathcal S^{\prime}\bot
(\mathcal T\cap \mathcal S)$. Namely, $\mathcal S^{\prime}=
(\mathcal T\cap \mathcal S)^{\bot}\cap\mathcal S$. Then $\dim \mathcal S^{\prime}=q$. Put
$\mathcal T^{\prime}=\pi_{\mathcal T}^{-1}(\mathcal S^{\prime})$. It is easy to see that
$\mathcal T^{\prime}\bot(\mathcal T\cap \mathcal S)$. Indeed, take  arbitrary vectors $v\in
\mathcal T\cap \mathcal S$, $v^{\prime}\in \mathcal S^{\prime}$, and a vector $w$,
such that $v^{\prime}+w=\pi_{\mathcal T}^{-1}(v^{\prime})$. In particular, $v^{\prime}+w\in\mathcal T$. Note that
$w\bot(\mathcal T\cap \mathcal S)$ because $\pi_{\mathcal T}$ is an orthogonal projection.
Then $\mathcal T\ni\pi_{\mathcal T}^{-1}(v^{\prime})=(v^{\prime}+w)\bot v$. Since
$v$ and $v^{\prime}$ are arbitrary, it follows that
$\mathcal T^{\prime}\bot (\mathcal T\cap \mathcal S)$. The non-degeneracy of $\pi_{\mathcal T}$ and
$\pi_{\mathcal T}^{-1}$ yields $\dim {\mathcal T}^{\prime}=q$.

Thus, we may regard the following mapping as the ``inverse'' projection $\pi_{\mathcal T}^{-1}$. To each $v\in \mathcal S$ it assigns the vector
$w\in\operatorname{span}\{v,\mathcal L\}\cap \mathcal T$ such that $w=(v+\mathcal L)\cap \mathcal T$. It is well defined  since $\mathcal L\bot \mathcal S$ and $\dim(\mathcal T\cap \mathcal L)=0$.

{\bf (iv)} Since $\mathcal L\cap \mathcal N$ is orthogonal to both $\mathcal S$ and $\mathcal T$, we
have
$$(v+\mathcal L)\cap \mathcal T=(v+\mathcal E)\cap \mathcal T$$
for $v\in \mathcal S$, where $\mathcal E\subset(\mathcal L\cap
\mathcal N)^{\bot}\cap \mathcal L$ is the minimal subset enjoying this property. Let us prove that $\mathcal E=(\mathcal L\cap
\mathcal N)^{\bot}\cap \mathcal L$. It is easy to see that if $w=(v+\mathcal L)\cap \mathcal T=v+u$ with $u\in \mathcal L$ then
$u=w-v\in\operatorname{span}\{\mathcal S,\mathcal T\}=(\mathcal L\cap \mathcal N)^{\bot}$. Consequently, $u\in(\mathcal L\cap
\mathcal N)^{\bot}\cap \mathcal L$, and $\dim \mathcal E\leq N-\widetilde{N}-l$.

Verify that $\dim \mathcal E\geq N-\widetilde{N}-l$. Assume on the contrary that $\mathcal E\neq
(\mathcal L\cap \mathcal N)^{\bot}\cap \mathcal L$, and consider $v_0\in \mathcal E^{\bot}\cap ((\mathcal L\cap
\mathcal N)^{\bot}\cap \mathcal L)$. Since $v_0\in\operatorname{span}\{\mathcal T,\mathcal S\}=(\mathcal L\cap \mathcal N)^{\bot}$, it follows that $v_0=v_0^{\mathcal T}+v_0^{\mathcal S}$, where $v_0^{\mathcal T}\in \mathcal T$ and $v_0^{\mathcal S}\in \mathcal S$. Take
$v=-v_0^{\mathcal S}\in\mathcal S$ and consider $v_0^{\mathcal T}=v+v_0\in \mathcal T$. By the assumption on $\mathcal E$ and because $\pi_{\mathcal T}$ is an epimorphism,
$v+v_0=v^{\mathcal S}+v^{\mathcal E}$, where $v^{\mathcal S}\in \mathcal S$ and $v^{\mathcal E}\in \mathcal E$. Hence, $v-v^{\mathcal S}=v^{\mathcal E}-v_0$, where $v-v^{\mathcal S}\in\mathcal S$ and $v^{\mathcal E}-v_0\in\mathcal L$. Since $\mathcal L\bot \mathcal S$, we arrive at a contradiction.

{\bf (v)} Similarly, we show that $\pi_{\mathcal N}^{-1}(w)=(w+{\mathcal F})\cap \mathcal N$,
where $\mathcal F=(\mathcal S\cap \mathcal T)^{\bot}\cap \mathcal S=\mathcal S^{\prime}$, and $\dim \mathcal F=\dim \mathcal E$.

{\bf (vi)} Take a cube of radius $r$ in $\mathcal S$ which is the direct product of a cube of radius $r$ in $\mathcal T\cap
\mathcal S$ and a cube of radius $r$ in $\mathcal S^{\prime}$. Then its image
equals the direct product of a cube of radius $r$ in $\mathcal T\cap
\mathcal S$ and a subset of $\mathcal T$. Hence, it suffices to calculate the measure distortion of $\pi_{\mathcal T}^{-1}$ on
$\mathcal S^{\prime}$. Similarly, we infer that it suffices to
calculate the measure distortion of $\pi_{\mathcal N}^{-1}$ on $\mathcal L^{\prime}$.

{\bf (vii)} Consider the orthogonal mapping $\Theta$ that is a rotation about $\operatorname{span}\{\mathcal N\cap \mathcal L, \mathcal T\cap \mathcal S\}$ such that
\begin{itemize}
\item $\Theta(\mathcal F)=\mathcal E$
\item $\Theta((\mathcal N\cap\mathcal L)^{\bot}\cap\mathcal L)=(\mathcal T\cap\mathcal S)^{\bot}\cap\mathcal S$.
\end{itemize}
Note that if we choose an orthonormal basis in $\mathbb R^N$ in the following order: $\operatorname{span}\{\mathcal N\cap \mathcal L, \mathcal T\cap \mathcal S\}$, $\operatorname{span}\{\mathcal F\}$, $\operatorname{span}\{(\mathcal N\cap\mathcal L)^{\bot}\cap\mathcal L\}$ then the matrix of $\Theta$ looks like
\begin{equation*}
\Theta=\left(%
\begin{array}{ccc}
 E & 0 & 0 \\
0  & 0 & E \\
0 & \widetilde{E} & 0 \\
\end{array}%
\right),
\end{equation*}
where 
$$\widetilde{E}=
\left(%
\begin{array}{ccccc}
-1 & 0 & 0 & 0 & 0 \\
0 & 1 & 0 & 0 & 0 \\
0 & 0 & 1 & 0 & 0 \\
\cdot & \cdot & \cdot & \cdot & \cdot \\
0 & 0 & 0 & 0 & 1 \\
\end{array}%
\right)
$$

Consequently,
${\pi_{\mathcal T}^{-1}}|_{\mathcal S^{\prime}}(B(0,r)\cap{\mathcal S^{\prime}})={\Theta\circ\pi_{\mathcal N}^{-1}}|_{\mathcal L^{\prime}}(B(0,r)\cap{\mathcal L^{\prime}})$ for any $r>0$,
and the measure distortions of these mappings coincide.

{\bf (viii)} Thus, the measure distortions of
$\pi_{\mathcal N}$ and $\pi_{\mathcal T}$ coincide.

Since both the Riemannian differential and $hc$-differential of $\theta_x$ at $x$
equal identity, by the chain rule the corresponding
``determinants'' coincide:
$$\mathcal D(D\varphi(x))=\mathcal D(D\psi(0))\text{ \ and \ }\mathcal D(\widehat{D}\varphi(x))=\mathcal D(\widehat{D}\psi(0)).$$
Thus, the measure distortion under $\pi_{\mathcal T}^{-1}$ equals
$$
\frac{\mathcal D({D}\psi(x))}{\mathcal D(\widehat{D}\psi(x))}=\frac{\mathcal D({D}\varphi(x))}{\mathcal D(\widehat{D}\varphi(x))}.
$$

{\bf Step III.} Verify that $\psi^{-1}(t)\cap\operatorname{Box}_2(0,r)$ is a subset
of an $o(r)$-neighborhood of
$\ker\widehat{D}\psi(0)\cap\operatorname{Box}_2(0,r(1+o(1)))$ with respect to $d^0_{2E}$. To this end, we prove that
$\partial(\psi^{-1}(t)\cap\operatorname{Box}_2(0,r))$ is a subset of the
$o(r)$-neighborhood of
$\partial(\ker\widehat{D}\psi(0)\cap\operatorname{Box}_2(0,r(1+o(1))))$
 with respect to $d_2$, where $o(1)$ is uniform in $x=\theta_x(0)$, $x\in\mathcal U\Subset\mathbb M$.

Indeed, consider $y\in\partial(\psi^{-1}(t)\cap\operatorname{Box}_2(0,r))$. We can represent this
point as $y=y_{\widehat{D}\psi}+y_{\ker\widehat{D}\psi}$,
where $y_{\widehat{D}\psi}\in(\ker\widehat{D}\psi(0))^{\bot}$ and
$y_{\ker\widehat{D}\psi}\in\ker\widehat{D}\psi(0)$. Then the definition of the $hc$-differentiability implies that
\begin{multline*}
o(r)=d_2^{\psi(0)}(\widehat{D}\psi(0)(y),\psi(y))=d_2^{\psi(0)}(\widehat{D}\psi(0)(y_{\widehat{D}\psi}),\psi(y))\\
=d_2^{\psi(0)}(\widehat{D}\psi(0)(y_{\widehat{D}\psi}),\psi(0)).
\end{multline*}
We have $d_2^{\psi(0)}(\psi(0),\widehat{D}\psi(0)(y_{\widehat{D}\psi}))=o(r)$, and
$d_2^0(0,y_{\widehat{D}\psi})=o(r)$, where $o(\cdot)$ is uniform in $y$, $x=\theta_x(0)$, $x\in\mathcal U\Subset\mathbb M$.  The inequality
\begin{equation}\label{gtineq}
d_{2E}^0(0,u+v)\leq d_{2E}^0(0,u)+cd_{2E}^0(0,v)
\end{equation}
implies that
$$
d_{2E}^0(0, y_{\ker\widehat{D}\psi})\leq d_{2E}^0(0,y)+cd_{2E}^0(0, y_{\widehat{D}\psi})=r+o(r),
$$
and hence, $d_{2E}^0(0, y_{\ker\widehat{D}\psi})=r(1+o(1))$ and
$d_2^0(0, y_{\ker\widehat{D}\psi})=r(1+o(1))$. Thus,
$\partial(\psi^{-1}(t)\cap\operatorname{Box}_2(0,r))$ is a subset of the
$o(r)$-neighborhood of
$\partial(\ker\widehat{D}\psi(0)\cap\operatorname{Box}_2(0,r(1+o(1))))$ with respect
to $d_2$.

{\bf Step IV.} Similarly considering the linear mapping
$L(y)=D\psi(0)y$, we infer that $\partial(\ker
D\psi(0)\cap\operatorname{Box}_2(0,r))$ is a subset of the $o(r)$-neighborhood of
$\partial(\ker\widehat{D}\psi(0)\cap\operatorname{Box}_2(0,r(1+o(1))))$ with respect to $d_{2E}^0$. Indeed, we have $\ker
D\psi(0)=L^{-1}(0)$ and $DL(0)=D\psi(0)$, and $L$ is $hc$-differentiable at $0$ since $D\psi(0)$ is horizontal, and
$\widehat{D}L(0)=\widehat{D}\psi(0)$.

{\bf Step V.} Since there exists a bijective linear mapping from
$\ker\widehat{D}\psi(0)\cap\operatorname{Box}_2(0,r)$ to
$\ker{D}\psi(0)\cap\operatorname{Box}_2(0,r(1+o(1)))$, by Step {\bf IV}
we infer that $\partial(\ker\widehat{D}\psi(0)\cap\operatorname{Box}_2(0,r))$ is a
subset of the $o(r)$-neighborhood of
$\partial(\ker{D}\psi(0)\cap\operatorname{Box}_2(0,r(1+o(1))))$.

Hence, $\partial(\psi^{-1}(t)\cap\operatorname{Box}_2(0,r))$ is
a subset of the $o(r)$-neigh\-borhood of
$\partial(\ker{D}\psi(0)\cap\operatorname{Box}_2(0,r(1+o(1))))$ with respect to $d_{2E}^0$.

{\bf Step VI.} In Steps {\bf  VII}~-- {\bf IX} we explain that
$$\partial(\ker{D}\psi(0)\cap\operatorname{Box}_2(0,r(1+o(1))))$$ is a subset of the
$o(r)$-neighborhood of $\partial(\psi^{-1}(t)\cap\operatorname{Box}_2(0,r))$ with respect to $d_{2E}^0$.
Moreover, similar arguments imply that the same is true regarding
the sets $\ker{D}\psi(0)\cap\operatorname{Box}_2(0,r(1+o(1)))$ and
$\psi^{-1}(t)\cap\operatorname{Box}_2(0,r)$. Indeed, it suffices to recall that $o(\cdot)$ in the previous steps are uniform in $r>0$ and $x=\theta_x(0)$.

{\bf Step VII.} In order to justify the result of Step {\bf VI}, we
construct a one-to-one mapping from
$\partial(\psi^{-1}(t)\cap\operatorname{Box}_2(x,r))$ to a subset of the
$o(r)$-neighborhood of
$\partial(\ker{D}\psi(0)\cap\operatorname{Box}_2(x,r(1+o(1))))$ (with respect to $d_{2E}^0$) lying in the plane $\ker D\psi(0)$.

First of all, observe that at the regular point $0$ we have $\ker
D\psi\cap(\ker\widehat{D}\psi)^{\bot}=\{0\}$ by Step {\bf II}, substep {\bf (iii)}.

From this we deduce the following assumption: without
loss of generality we may assume that $D\psi(z)(v)\geq\beta$, where $z\in
U\supset\operatorname{Box}_2(0,r)$,  $\beta>0$, for $v\in
(\ker\widehat{D}\psi(0))^{\bot}$ with $|v|=1$.

This implies that in a neighborhood of a regular point we can expand each
point $y$ uniquely as $y=y_{\ker
D\psi}+y_{\widehat{D}\psi}$, where $y_{\ker
D\psi}\in{\ker
D\psi}$ and $y_{\widehat{D}\psi}\in (\ker\widehat{D}\psi)^{\bot}$. Suppose that
$y\in\psi^{-1}(t)\cap\partial\operatorname{Box}_2(0,r)$ and consider the
mapping $\xi:y\mapsto y_{\ker D\psi}$.

Verify that $d_2(0,\xi(y))=r(1+o(1))$. To this end, we prove that
$d_2^0(0,y-\xi(y))=o(r)$ and then apply
{\eqref{gtineq}} for $d^0_{2E}$.

Represent $y$ as $y=y_{\ker
D\psi}+y_{\widehat{D}\psi}=v_{\ker\widehat{D}\psi}+v_{\widehat{D}\psi}+y_{\widehat{D}\psi}$, where $y_{\ker
D\psi}=v_{\ker\widehat{D}\psi}+v_{\widehat{D}\psi}$.
Since $d_{2E}^0(0,y)=r$, Step {\bf III} implies that
$d_{2E}^0(0,v_{\ker \widehat{D}\psi})=r+o(r)$ and $d_{2E}^0(0,v_{\widehat{D}\psi}+y_{\widehat{D}\psi})=o(r)$. Moreover, by Step {\bf IV} we have
$$
d_{2E}^0(0,y_{\ker D\psi})=(1+o(1))d_{2E}^0(0, v_{\ker\widehat{D}\psi}),\quad d_{2E}^0(0,v_{D\psi})=o(1)d_{2E}^0(0, v_{\ker\widehat{D}\psi}),
$$
and taking into account the generalized triangle inequality for $d^0_{2E}$, we conclude that $d_{2E}^0(0,y_{\widehat{D}\psi})=o(r)$. Finally, we have
$d_2^0(0,y-\xi(y))=o(r)$.

{\bf Step VIII.} On this step, we show that $\xi$ is a bijective
mapping. It is easy to see that $\xi$ is injective. Indeed, it follows from $D\psi\langle(\ker\widehat{D}\psi)^{\bot}\cap\mathbb S^{N-1}\rangle\geq \beta>0$.

To verify that it is also surjective, we show that the mapping $\eta:\psi^{-1}\cap\operatorname{Box}_2(0,2r)\to\ker D\psi(0)$ is bi-Lipschitz. Here, $\eta$ assigns to each $y\in\psi^{-1}(t)\cap\operatorname{Box}_2(0,r)$, the vector $y_{\ker D\psi}$, where $y=y_{\ker
D\psi}+y_{\widehat{D}\psi}$, $y_{\ker
D\psi}\in{\ker
D\psi}$, and $y_{\widehat{D}\psi}\in (\ker\widehat{D}\psi)^{\bot}$. It is clear that $\eta|_{\psi^{-1}(t)\cap\partial\operatorname{Box}_2(0,r)}=\xi$.

Since $\eta$ is a projection, it is Lipschitz with the Lipschitz constant being equal to~1.
Verify that $|\eta(y_1)-\eta(y_2)|\geq K|y_1-y_2|$ for some $K>0$ for all $y_1, y_2\in \psi^{-1}(t)\cap\operatorname{Box}_2(0,2r)$.

Suppose that it is not so; thus, for every $\varepsilon>0$ there exist $y_1, y_2\in \psi^{-1}(t)\cap\operatorname{Box}_2(0,2r)$ such that $|\eta(y_1)-\eta(y_2)|<\varepsilon|y_1-y_2|$. Since $y_i={y_i}_{\ker
D\psi}+{y_i}_{\widehat{D}\psi}$, ${y_i}_{\ker
D\psi}\in{\ker
D\psi}$ and ${y_i}_{\widehat{D}\psi}\in (\ker\widehat{D}\psi)^{\bot}$, $i= 1,2$, we infer that
$$
|\eta(y_1)-\eta(y_2)|<\varepsilon|y_1-y_2|\leq\varepsilon|\eta(y_1)-\eta(y_2)|+\varepsilon|{y_1}_{\widehat{D}\psi}-{y_2}_{\widehat{D}\psi}|,
$$
and
$$
|{y_1}_{\widehat{D}\psi}-{y_2}_{\widehat{D}\psi}|\geq\frac{1-\varepsilon}{\varepsilon}|\eta(y_1)-\eta(y_2)|.
$$
Since $\ker D\psi(0)\cap(\ker\widehat{D}\psi(0))^{\bot}=0$, the norm $|y_1-y_2|_{\eta}=|\eta(y_1)-\eta(y_2)|+|{y_1}_{\widehat{D}\psi}-{y_2}_{\widehat{D}\psi}|$ is equivalent to the Euclidean norm.
Furthermore,
$$
\psi(y_1)=\psi(y_2)+D\psi(y_2)(y_1-y_2)+o(|y_1-y_2|),
$$ and $D\psi(y_2)(y_1-y_2)=o(|y_1-y_2|)$. Since $\eta$ is a linear mapping, we have ${y_1}_{\widehat{D}\psi}-{y_2}_{\widehat{D}\psi}=(y_1-y_2)_{\widehat{D}\psi}$ and $\eta(y_1)-\eta(y_2)=\eta(y_1-y_2)$.

On the one hand, by the definition of $\eta$ we have
\begin{multline*}
D\psi(y_2)(y_1-y_2)=D\psi(0)(y_1-y_2)+o(|y_1-y_2|)\\
=D\psi(0)((y_1-y_2)_{\widehat{D}\psi})+o(|y_1-y_2|),
\end{multline*}
and $D\psi(0)((y_1-y_2)_{\widehat{D}\psi})=o(|y_1-y_2|)$. On the other hand,
$$
|{y_1}_{\widehat{D}\psi}-{y_2}_{\widehat{D}\psi}|\geq \Bigl(1-\Bigl(\frac{1-\varepsilon}{\varepsilon}+1\Bigr)^{-1}\Bigr)|y_1-y_2|_{\eta}\geq L \Bigl(1-\Bigl(\frac{1-\varepsilon}{\varepsilon}+1\Bigr)^{-1}\Bigr)|y_1-y_2|,
$$
where $L>0$ depends only on $U$.

These relations lead to a contradiction with
$$
D\psi(0)\langle(\ker\widehat{D}\psi(0))^{\bot}\rangle\geq\beta>0.
$$

Thus, $\eta$ is a bi-Lipschitz mapping, and each $y\in\psi^{-1}(t)\cap\partial\operatorname{Box}_2(0,r)$ has a unique preimage. Thus, $\xi$ is bijective.

{\bf Step IX.} Step {\bf VIII} implies that $\xi$ is
invertible, and the previous steps imply that the $d_2$-distortion
of $\xi^{-1}$ is $(1+o(1))$ with respect to $0$. The estimate for
$d_2(0, y-\eta^{-1}(y))$ is proved in Step {\bf VII}. Observe that $o(1)$
depends on the convergence of $o(1)$ to $0$ in the equality
$$
d_2^{\psi(0)}(\widehat{D}\psi(0)y,\psi(y))=o(1)\cdot d_2(0,y).
$$
Thus, $o(1)$ is uniform in $y$. Consequently,
$\psi^{-1}(t)\cap\partial\operatorname{Box}_2(0,r)$ is a subset of the
$o(r)$-neighborhood of $\ker
D\psi(0)\cap\partial\operatorname{Box}_2(0,r(1+o(1)))$ with respect to $d_2$, and
conversely, $\ker D\psi(0)\cap\partial\operatorname{Box}_2(0,r)$ is a subset of the
$o(r)$-neighborhood of $\psi^{-1}(t)\cap\partial\operatorname{Box}_2(0,r)$ with
respect to $d^0_{2E}$ (see Step {\bf VII}). The same is true regarding the sets
$\psi^{-1}(t)\cap\operatorname{Box}_2(0,r)$ and $\ker
D\psi(0)\cap\operatorname{Box}_2(0,r(1+o(1)))$.

Note that the extension $\eta$ of $\xi$ onto
$\psi^{-1}(t)\cap\operatorname{Box}_2(0,r)$ has the $d_2$-distortion with respect
to $0$ equal to $1+o(1)$ as well. In Step
{\bf VIII}, we prove that this extension is bijective.

{\bf Step X.} On this step, we show that the classical, that is, ``Riemannian''
distortion of $\eta$ (and of $\xi$) is also $1+o(1)$.

Indeed, since $\ker D\psi(0)\cap(\ker\widehat{D}\psi(0))^{\bot}=\{0\}$ and $D\psi\langle(\ker\widehat{D}\psi(0))^{\bot}\rangle\geq\beta>0$ on $U$, it is easy to see that $\rho(0,y_{\widehat{D}\psi})=o(\rho(0,y))$. Hence,
$\rho(0,y_{\widehat{D}\psi})=o(\rho(0,y_{\ker{D}\psi}))$, and the
$\rho$-distortion of both $\eta$ and $\eta^{-1}$ equals $1+o(1)$.

Since $o(\cdot)$ are uniform in the definitions of Riemannian differentiability and $hc$-differentiability,
$o(\cdot)$ is uniform (both in the case of $d_2$- and
$\rho$-distortion).

Thus, claim ${\bf I}$ is proved.

{\bf Step XI.} Using the results of all previous steps, we can
consider the mapping $\eta^{-1}$ (that is, the extension of $\xi^{-1}$) from $\ker D\psi\cap\operatorname{Box}_2(0,r(1+o(1)))$
to $\psi^{-1}(t)\cap\operatorname{Box}_2(0,r)$ (see Steps {\bf VII}--{\bf IX}). The
established properties of~$\eta$ imply that the measure of
$\psi^{-1}(t)\cap\operatorname{Box}_2(0,r)$ equals that of $\ker
D\psi\cap\operatorname{Box}_2(0,r)$ up to a factor of $1+o(1)$, that is, it equals
$C(1+o(1))r^{\nu-\tilde{\nu}}$, where $C$ is obtained above.

Indeed, $\eta$ is a bi-Lipschitz $C^1$-mapping since it
is a projection of a $C^1$-surface in a nontangent direction.
Consequently, the inverse mapping $\eta^{-1}$ is also a bi-Lipschitz
$C^1$-mapping. Moreover, since it is differentiable, it is
also metrically differentiable, and its metric differential equals
$1+o(1)$ on every direction. The $C^1$-mapping $\eta^{-1}$ satisfies ${\cal H}^{N-\widetilde{N}}(\eta^{-1}(A))={\cal J}(\eta^{-1},0)\cdot{\cal
H}^{N-\widetilde{N}}(A)\cdot(1+o(1))$, where
$A=\eta(\psi^{-1}(t)\cap\operatorname{Box}_2(0,r))$. We can calculate the Jacobian of $\eta^{-1}$ via its metric differential {\cite{kirch1, Km1, Km3, Km5, Km6}}:
$$
{\cal J}(\eta^{-1},0)=\sigma_{N-\widetilde{N}}\Bigl[\int\limits_{{\mathbb
S}^{N-\widetilde{N}-1}}(1+o(1))^{\widetilde{N}-N}\,d{\cal
H}^{N-\widetilde{N}}(u)\Bigr]^{-1}=(1+o(1)).
$$
Thus, ${\cal H}^{N-\widetilde{N}}(\eta^{-1}(A))={\cal
H}^{N-\widetilde{N}}(A)(1+o(1))$. By taking into account a remark in Step~{\bf V} substep~{\bf (ii)} of the proof of Theorem~{\ref{tang_meas}} concerning regular points, we infer ${\cal
H}^{N-\widetilde{N}}(A)=\prod\limits_{k=1}^{M}\omega_{n_k-\tilde{n}_k}\cdot
\frac{\mathcal D({D}\varphi(x))}{\mathcal D(\widehat{D}\varphi(x))}r^{\nu-\tilde{\nu}}(1+o(1))$.

{\bf Step XII.} Consider the restriction of
$\theta_x$ to $\psi^{-1}(t)\cap\operatorname{Box}_2(0,r)$.
Observe that
$\theta_x(\psi^{-1}(t)\cap\operatorname{Box}_2(0,r))=\varphi^{-1}(t)\cap\operatorname{Box}_2(x,r)$.

The ${\cal H}^{N-\widetilde{N}}$-measure distortion under this mapping
equals $\widetilde{\mathcal D}(g|_{\ker D\varphi(x)})$, and $c_{\operatorname{Riem}}=\widetilde{\mathcal D}(g|_{\ker
D\varphi(x)})$.

Claim {\bf II} is proved.

All $o(\cdot)$ in both claims {\bf I} and {\bf II} are uniform in the radius $r$ and $x=\theta_x(0)$, $x\in\mathcal U\Subset\mathbb M$.

The proof of the theorem is complete.
\end{proof}

\begin{defn} The (spherical) Hausdorff measure of a set
$A\subset\varphi^{-1}(t)$ (constructed with respect to a sub-Rie\-man\-nian
(quasi)metric $d$ and sub-Rie\-man\-nian balls ${B}_d$ in~$d$) equals
$$
{\cal
H}^{\nu-\tilde{\nu}}(A)=\omega_{\nu-\tilde{\nu}}\lim\limits_{\delta\to0}\inf
\Bigl\{\sum\limits_{i\in\mathbb
N}r_i^{\nu-\tilde{\nu}}:\bigcup\limits_{i\in\mathbb
N}{B_d}(x_i,r_i)\supset A, x_i\in A, r_i\leq\delta, i\in\mathbb N\Bigr\}.
$$
\end{defn}

\begin{property}[{\cite{vk_birk}}]
The quasimetric $d_2$ and the metric $d_{cc}$ (see Definitions~{\ref{defhoriz}} and~{\ref{defdcc}}) are locally equivalent.
\end{property}

\begin{lem}\label{lemma_zero} Given a set $A\subset\varphi^{-1}(t)$ of
$\mathcal H^{N-\widetilde{N}}$-measure zero consisting of regular points
and $\varepsilon>0$, there exists a covering of $A$ by the ``balls''
$\{\operatorname{Box}_2(x_{j},r_{j})\cap\varphi^{-1}(t)\}_{j\in\mathbb
N}$, where $x_j\in\varphi^{-1}(t)$, $j\in\mathbb N$, the sum of whose $\mathcal H^{N-\widetilde{N}}$-measures is
less than $\varepsilon$.
\end{lem}

\begin{proof} Fix $\varepsilon>0$. Represent $A$ as the union of some subsets
$A_l\subset A$, $l\in\mathbb N$, lying at positive distances from
the set $\chi\cap\varphi^{-1}(t)$. Without loss of generality we
consider a set $A_l$ instead of $A$.
Suppose that $A_l$ satisfies the stated conditions. Observe that for this set there
exists a collection of ``balls''
$\{B(x_j,r_j)\cap\varphi^{-1}(t)\}_{j\in\mathbb N}$, where $x_j\in\varphi^{-1}(t)$ and $B(x_j,r_j)$ are Riemannian balls, $j\in\mathbb N$, the
sum of whose $\mathcal H^{N-\widetilde{N}}$-measures is less than
$\tilde{\varepsilon}$, where $\tilde{\varepsilon}$ is
determined by $\varepsilon$ (we specify the exact expression for
$\tilde{\varepsilon}$ below). Fix $j\in\mathbb N$. For $B(x_j,r_j)\cap\varphi^{-1}(t)$, there exists a collection
$$
\{{B}_{cc}(x,r)\cap\varphi^{-1}(t):x\in A_l\cap\varphi^{-1}(t),
{B}_{cc}(x,r)\cap\varphi^{-1}(t)\subset B(x_j,r_j),
r>0\}.
$$
By the $5r$-covering lemma (see {\cite{F}} for instance)
there exists a countable family of disjoint ``balls''
$\{{B}_{cc}(x_{j_k},r_{j_k})\cap\varphi^{-1}(t)\}$ such that
$$
\bigcup\limits_{k\in\mathbb N}{B}_{cc}(x_{j_k},5r_{j_k})\supset
A_l\cap B(x_j,r_j).
$$
(Here we use the metric $d_{cc}$ since the $5r$-covering lemma is established for metrics, and we cannot say for sure whether it holds for quasimetrics). Since $d_2$ and $d_{cc}$ are locally equivalent, there exist
some constants $0<C_1, C_2<\infty$ such that
$\operatorname{Box}_{2}(x,C_1r)\subset{B}_{cc}(x,r)\subset
\operatorname{Box}_{2}(x,C_2r)$ for sufficiently small $r>0$ and $x$
in some sufficiently small neighborhood. Consequently, there
exists a disjoint collection
$\{\operatorname{Box}_{2}(x_{j_k},C_1r_{j_k})\}$ such that
$$
\bigcup\limits_{k\in\mathbb
N}\operatorname{Box}_{2}(x_{j_k},5C_2r_{j_k})\supset A_l\cap
B(x_j,r_j).
$$
This implies that
\begin{multline*}
\sum\limits_{k\in\mathbb N}\mathcal
H^{N-\widetilde{N}}(\operatorname{Box}_{2}(x_{j_k},5C_2r_{j_k}))\leq
C(5,C_1,C_2)\sum\limits_{k\in\mathbb N}\mathcal
H^{N-\widetilde{N}}(\operatorname{Box}_{2}(x_{j_k},C_1r_{j_k}))\\
\leq C(5,C_1,C_2)\mathcal
H^{N-\widetilde{N}}(B(x_j,r_j)\cap\varphi^{-1}(t)).
\end{multline*}
Then
$\{\operatorname{Box}_{2}(x_{j_k},5C_2r_{j_k})\}_{j,k\in\mathbb
N}$ is a reqired collection,  and the sum of the $\mathcal
H^{N-\widetilde{N}}$-measures of these ``balls'' is at most
$C(5,C_1,C_2)\tilde{\varepsilon}<\varepsilon$.
\end{proof}

\begin{cor}\label{abscont} For each $\mathcal H^{N-\widetilde{N}}$-measure zero subset of $\varphi^{-1}(t)$ consisting of regular points
we have $\mathcal H^{\nu-\tilde{\nu}}=0$.
\end{cor}

The {proof} follows directly from the definition of an $\mathcal H^{\nu-\tilde{\nu}}$-negligible set.

\begin{lem}\label{opt_cov} Consider a regular point $x\in\mathbb M$.
Then given a sufficiently small $r>0$ with
$\operatorname{Box}_2(x,r)\cap\varphi^{-1}(\varphi(x))\cap\chi=\varnothing$ and $\varepsilon>0$, there exists a covering of $\operatorname{Box}_2(x,r)\cap\varphi^{-1}(\varphi(x))$ by the sets
$\operatorname{Box}_2(x_i,r_i)\cap\varphi^{-1}(\varphi(x))$ with $x_i\in\varphi^{-1}(\varphi(x))$ and
$\operatorname{Box}_2(x_i,r_i)\subset\operatorname{Box}_2(x,r)$, $i\in\mathbb N$,
such that
$$
\sum\limits_{i\in\mathbb N}\mathcal
H^{N-\widetilde{N}}(\operatorname{Box}_2(x_i,r_i)\cap\varphi^{-1}(\varphi(x)))
<\mathcal
H^{N-\widetilde{N}}(\operatorname{Box}_2(x,r)\cap\varphi^{-1}(\varphi(x)))+\varepsilon.
$$
\end{lem}

\begin{proof} Without loss of generality we assume that $r>0$ satisfies
$$
\operatorname{dist}(\operatorname{Box}_2(x,r)\cap\varphi^{-1}(\varphi(x)),\chi)>0.
$$
Since the measure $\mathcal H^{N-\widetilde{N}}$ is doubling on
$\operatorname{Box}_2(y,s)\cap\varphi^{-1}(\varphi(x))$ for $y\in\varphi^{-1}(\varphi(x))$ and sufficiently small $s$, in view of the Vitali covering theorem there exists
a collection of ``balls''
$\{\operatorname{Box}_2(x_{i_j},r_{i_j})\cap\varphi^{-1}(\varphi(x))\}
_{j\in\mathbb N}$ with $x_{i_j}\in\varphi^{-1}(\varphi(x))$ and $\operatorname{Box}_2(x_{i_j},r_{i_j})\cap\varphi^{-1}(\varphi(x))\subset\operatorname{Box}_2(x,r)\cap\varphi^{-1}(\varphi(x))$, $j\in\mathbb N$, such that
$$
\mathcal
H^{N-\widetilde{N}}(\operatorname{Box}_2(x,r)\cap\varphi^{-1}(\varphi(x)))
=\sum\limits_{j\in\mathbb N}\mathcal
H^{N-\widetilde{N}}(\operatorname{Box}_2(x_{i_j},r_{i_j})\cap\varphi^{-1}(\varphi(x))).
$$
By Lemma {\ref{lemma_zero}}, the $\mathcal
H^{N-\widetilde{N}}$-negligible set
$$\operatorname{Box}_2(x,r)\cap\varphi^{-1}(\varphi(x))\setminus\bigcup\limits_{j\in\mathbb
N}\operatorname{Box}_2(x_{i_j},r_{i_j})\cap\varphi^{-1}(\varphi(x)),$$
admits a covering by the collection
$\{\operatorname{Box}_2(x_{i_k},r_{i_k})\cap\varphi^{-1}(\varphi(x))\}$ of ``balls'' the sum of whose $\mathcal H^{N-\widetilde{N}}$-measures is less
than the given $\varepsilon>0$. The proof is complete.
\end{proof}

\medskip

Recall that $D_{\mu_1}{\mu_2}(y)$ stands for the
derivative of a measure $\mu_2$ with respect to a measure $\mu_1$
at~$y$:
$$
D_{\mu_1}{\mu_2}(y)=\lim\limits_{r\to0}\frac{\mu_2(B(y,r))}{\mu_1(B(y,r))}.
$$

\begin{thm}[The Lebesgue Differentiation of Measures  on Level
Sets]\label{derivativemeas} The Hausdorff measure $\mathcal H^{\nu-\tilde{\nu}}$ of $\operatorname{Box}_2(x,r)\cap\varphi^{-1}(\varphi(x))$, where $x$
is a regular point, and
$\operatorname{dist}(\operatorname{Box}_2(x,r)\cap\varphi^{-1}(\varphi(x)),\chi)>0$,
asymptotically equals $\omega_{\nu-\tilde{\nu}}r^{\nu-\tilde{\nu}}$:
$$
{\cal
H}^{\nu-\tilde{\nu}}(\operatorname{Box}_2(x,r)\cap\varphi^{-1}(\varphi(x)))=\omega_{\nu-\tilde{\nu}}r^{\nu-\tilde{\nu}}(1+o(1)).
$$
The
derivative $D_{{\cal H}^{N-\widetilde{N}}}{\cal H}^{\nu-\tilde{\nu}}(x)$
equals
$$
\frac{1}{\widetilde{\mathcal D}(g|_{\ker
D\varphi(x)})}\cdot\frac{\omega_{\nu-\tilde{\nu}}}{\prod\limits_{k=1}^{M}{\omega_{n_k-\tilde{n}_k}}}\cdot
\frac{\mathcal D(\widehat{D}\varphi(x))}{\mathcal D({D}\varphi(x))}.
$$
\end{thm}

\begin{proof} Consider the intersection of the sub-Riemannian ball of a sufficiently small radius $r>0$ centered at $x$ and the level set
$\varphi^{-1}(\varphi(x))$. Fix $\delta>0$ and some covering $\{\operatorname{Box}_2(y_i,r_i)\}_{i\in\mathbb N}$ of
this intersection by
sub-Riemannian balls as in the definition of the set function $\mathcal
H^{\nu-\tilde{\nu}}_{\delta}$. Recall that for a (quasi)metric $d$ we have
$$
{\cal
H}^{\nu-\tilde{\nu}}_{\delta}(A)=\omega_{\nu-\tilde{\nu}}\inf
\Bigl\{\sum\limits_{i\in\mathbb
N}r_i^{\nu-\tilde{\nu}}:\bigcup\limits_{i\in\mathbb
N}{B_d}(x_i,r_i)\supset A, x_i\in A, r_i\leq\delta, i\in\mathbb N\Bigr\}
$$
Then, setting
$${\alpha(y)}=\lim\limits_{r\to0}\frac{\omega_{\nu-\tilde{\nu}}r^{\nu-\tilde{\nu}}}{{\cal
H}^{N-\widetilde{N}}(\operatorname{Box}_2(y,r)\cap\varphi^{-1}(\varphi(y)))},$$
by Theorem
{\ref{Leb_meas}}, we have
\begin{align*}
\omega_{\nu-\tilde{\nu}}\sum\limits_{i\in\mathbb N}r_i^{\nu-\tilde{\nu}}
&=\sum\limits_{i\in\mathbb N}[\alpha(y_i)+\delta(y_i, r_i)]{\cal
H}^{N-\widetilde{N}}(\operatorname{Box}_2(y_i,r_i)\cap\varphi^{-1}(\varphi(x)))\\
&=[\alpha(x)+\Delta(x,r)]\sum\limits_{i\in\mathbb N}{\cal
H}^{N-\widetilde{N}}(\operatorname{Box}_2(y_i,r_i)\cap\varphi^{-1}(\varphi(x))),
\end{align*}
where $\delta(y_i, r_i)\to0$ as $r_i\to0$ uniformly in $i\in\mathbb N$, and $\Delta(x,r)\to0$ as $r\to0$, since $g$, $D\varphi$ and
$\widehat{D}\varphi$ are continuous, and by Theorem
{\ref{Leb_meas}} as well. For the fixed $x$ and $\varepsilon>0$,
we can choose a sufficiently small radius $r>0$ to satisfy
$|\Delta(x,r)|<\varepsilon\cdot\alpha(x)$. Consequently,
\begin{multline*}
(1-\varepsilon)\alpha(x)\sum\limits_{i\in\mathbb N}{\cal
H}^{N-\widetilde{N}}(\operatorname{Box}_2(y_i,r_i)\cap\varphi^{-1}(\varphi(x)))\\
\leq\omega_{\nu-\tilde{\nu}}\sum\limits_{i\in\mathbb
N}r_i^{\nu-\tilde{\nu}}
\leq(1+\varepsilon)\alpha(x)\sum\limits_{i\in\mathbb N}{\cal
H}^{N-\widetilde{N}}(\operatorname{Box}_2(y_i,r_i)\cap\varphi^{-1}(\varphi(x))).
\end{multline*}
Note that for each $\delta>0$ if the covering in the
definition of $\mathcal H^{\nu-\tilde{\nu}}_{\delta}$ is nearly
``optimal'' (i.~e., $\sum\limits_{i\in\mathbb
N}r_i^{\nu-\tilde{\nu}}$ is nearly minimal), then
$\sum\limits_{i\in\mathbb N}{\cal
H}^{N-\widetilde{N}}(\operatorname{Box}_2(y_i,r_i)\cap\varphi^{-1}(\varphi(x)))$ is also
nearly minimal. Since the
${\cal H}^{N-\widetilde{N}}$-measure is countably additive, by Lemma {\ref{opt_cov}} the
infimum of the values of the last sum equals $\mathcal
H^{N-\widetilde{N}}(\operatorname{Box}_2(x,r))$.

Hence,
$$
{\cal
H}^{\nu-\tilde{\nu}}(\operatorname{Box}_2(x,r)\cap\varphi^{-1}(\varphi(x)))=\omega_{\nu-\tilde{\nu}}r^{\nu-\tilde{\nu}}(1+o(1)),
$$
where $o(1)\to0$ as $r\to0$, and
\begin{align}\label{derivative}
D_{{\cal H}^{N-\widetilde{N}}}{\cal
H}^{\nu-\tilde{\nu}}(x)&={\alpha(x)}=\lim\limits_{r\to0}\frac{\omega_{\nu-\tilde{\nu}}r^{\nu-\tilde{\nu}}}{{\cal
H}^{N-\widetilde{N}}(\operatorname{Box}_2(x,r)\cap\varphi^{-1}(\varphi(x)))}\\
&=\frac{1}{\widetilde{\mathcal D}(g|_{\ker
D\varphi(x)})}\cdot\frac{\omega_{\nu-\tilde{\nu}}}{\prod\limits_{k=1}^{M}{\omega_{n_k-\tilde{n}_k}}}\cdot
\frac{\mathcal D(\widehat{D}\varphi(x))}{\mathcal D({D}\varphi(x))}.
\end{align}
The proof is complete.
\end{proof}

\medskip

The last result motivates the next definition.

\begin{defn} The {\it sub-Riemannian coarea factor} equals
$$
{\cal J}^{SR}_{\widetilde{N}}(\varphi,x)=\mathcal D(\widehat{D}\varphi(x))\cdot
\frac{\omega_{N}}{\omega_{\nu}}\frac{\omega_{\tilde{\nu}}}{\omega_{\widetilde{N}}}
\frac{\omega_{\nu-\tilde{\nu}}}{\prod\limits_{k=1}^{M}\omega_{n_k-\tilde{n}_k}}.
$$
\end{defn}

\begin{rem}\label{exprcf}
By {\eqref{derivative}}, we have
$${\mathcal J}^{SR}_{\widetilde{N}}(\varphi,x)=\mathcal J_{\widetilde{N}}(\varphi,x)\cdot D_{{\cal H}^{N-\widetilde{N}}}{\cal
H}^{\nu-\tilde{\nu}}(x)\cdot\frac{\widetilde{\mathcal D}(g(x))}{\widetilde{\mathcal D}(\tilde{g}(\varphi(x)))}\cdot \frac{\omega_{N}}{\omega_{\nu}}\frac{\omega_{\tilde{\nu}}}{\omega_{\widetilde{N}}}
$$
since
$$
\mathcal J_{\widetilde{N}}(\varphi,x)={{\mathcal D}({D}\varphi(x))}\cdot\frac{\widetilde{\mathcal D}(\tilde{g}(\varphi(x)))\widetilde{\mathcal D}(g|_{\ker
D\varphi(x)})}{\widetilde{\mathcal D}(g(x))}.
$$
\end{rem}

\begin{cor}\label{recov_meas} Theorem ${\ref{derivativemeas}}$,
Corollary ${\ref{abscont}}$, and the Lebesgue differentiation theorem $($see, for instance, {\rm\cite{vu}}$)$ imply that
every measurable subset of a regular set $A\subset(\mathbb
D\cap\varphi^{-1}(t))$, $t\in\widetilde{\mathbb M}$, satisfies
$$
\mathcal H^{\nu-\tilde{\nu}}(A)=\int\limits_AD_{{\cal
H}^{N-\widetilde{N}}}{\cal H}^{\nu-\tilde{\nu}}(x)\,d\mathcal H^{N-\widetilde{N}}(x).
$$
\end{cor}

\begin{thm}[Local coarea formula for the set $\mathbb D$]\label{coaread} If $\varphi\in C^1(\mathbb M, \widetilde{\mathbb M})$ then
$$
\int\limits_{\mathbb D}{\cal
J}^{SR}_{\widetilde{N}}(\varphi,x)\,d{\cal
H}^{\nu}(x)=\int\limits_{\widetilde{\mathbb M}}\,d{\cal
H}^{\tilde{\nu}}(t)\int\limits_{\varphi^{-1}(t)\cap\mathbb
D}\,d{\cal H}^{\nu-\tilde{\nu}}(u).
$$
\end{thm}

\begin{proof} The Riemannian coarea
formula {\cite{F}} yields
\begin{multline*}
\int\limits_{\mathbb D}{\cal
J}^{SR}_{\widetilde{N}}(\varphi,x)\,d{\cal
H}^{\nu}(x)=\int\limits_{\mathbb D}\frac{{\cal
J}^{SR}_{\widetilde{N}}(\varphi,x)}{\mathcal D(g(x))}\,d{\cal
H}^{N}(x)\\
=\int\limits_{\widetilde{\mathbb M}}\,d{\cal
H}^{\widetilde{N}}(t)\int\limits_{\varphi^{-1}(t)\cap\mathbb D}\frac{{\cal
J}^{SR}_{\widetilde{N}}(\varphi,u)}{{\cal J}_{\widetilde{N}}(\varphi,u)\widetilde{\mathcal D}(g(u))}\,d{\cal
H}^{N-\widetilde{N}}(u).
\end{multline*}

Calculating the measure derivatives in $\widetilde{\mathbb M}$, we obtain
\begin{multline*}
\int\limits_{\widetilde{\mathbb M}}\,d{\cal
H}^{\widetilde{N}}(t)\int\limits_{\varphi^{-1}(t)\cap\mathbb D}\frac{{\cal
J}^{SR}_{\widetilde{N}}(\varphi,u)}{{\cal J}_{\widetilde{N}}(\varphi,u)\widetilde{\mathcal D}(g(u))}\,d{\cal
H}^{N-\widetilde{N}}(u)\\
=\int\limits_{\widetilde{\mathbb M}}\,d{\cal
H}^{\tilde{\nu}}(t)\int\limits_{\varphi^{-1}(t)\cap\mathbb D}\frac{{\cal
J}^{SR}_{\widetilde{N}}(\varphi,u)\widetilde{\mathcal D}(\tilde{g}(t))}{{\cal
J}_{\widetilde{N}}(\varphi,u)\widetilde{\mathcal D}(g(u))}\,d{\cal H}^{N-\widetilde{N}}(u).
\end{multline*}
Finally, taking Remark {\ref{exprcf}} and Corollary {\ref{recov_meas}} into account, we infer that
\begin{multline*}
\int\limits_{\widetilde{\mathbb M}}\,d{\cal
H}^{\tilde{\nu}}(t)\int\limits_{\varphi^{-1}(t)\cap\mathbb D}\frac{{\cal
J}^{SR}_{\widetilde{N}}(\varphi,u)\widetilde{\mathcal D}(\tilde{g}(t))}{{\cal
J}_{\widetilde{N}}(\varphi,u)\widetilde{\mathcal D}(g(u))}\,d{\cal
H}^{N-\widetilde{N}}(u)\\
=\int\limits_{\widetilde{\mathbb M}}\,d{\cal
H}^{\tilde{\nu}}(t)\int\limits_{\varphi^{-1}(t)\cap\mathbb D}\,d{\cal
H}^{\nu-\tilde{\nu}}(u).
\end{multline*}
The proof is complete.
\end{proof}

\section{The Characteristic Set}

The goal of this section is to prove

\begin{thm}[see also {\cite{Km7}}]\label{charsetth}
If Assumption ${\ref{assumpphi}}$ is fulfilled then
$$
\mathcal H^{\nu-\tilde{\nu}}(\chi_t)=0
$$
for almost all $t\in\widetilde{\mathbb M}$ $($with respect to both $\mathcal H^{\widetilde{N}}$ and $\mathcal H^{\tilde{\nu}}$$)$.
\end{thm}

Recall that we denote $\varphi^{-1}(t)\cap\chi$ by~$\chi_t$ for $t\in\widetilde{\mathbb M}$.

We prove the results of this section under the following assumptions.

\begin{assump}
Assume that the basis vector fields $\{X_i\}_{i=1}^N$ in the preimage are of class $C^{M+1}$, and $\{\widetilde{X}_j\}_{j=1}^{\widetilde{N}}$ in the image are of class $C^{1,\tilde{\alpha}}$, $\tilde{\alpha}>0$. Moreover, assume that $\varphi\in C^{M+2}(\mathbb M, \widetilde{\mathbb M})$, where $M$ is the depth of $\mathbb M$ (see Definition {\ref{carnotmanifold}}).
\end{assump}

\begin{lem}\label{equivqm}
Let $\mathbb M$ be a Carnot manifold with $C^{1,\alpha}$-smooth basis vector fields, $\alpha>0$. From the basis vector fields
$\{X_i\}_{i=1}^N$ construct linearly independent vector fields $\{Y_j\}_{j=1}^{N}$ as
$Y_j(x)=\sum\limits_{i=1}^{\dim
H_{\operatorname{deg}X_j}}a_{ji}(x)X_i$, where $a_{ji}(x)$ are $C^{1,\omega}$-smooth functions for all $i$ and $j$, $\omega>0$. Then the quasimetric of type $d_{2}$ constructed with respect to the new basis $\{Y_j\}_{j=1}^{N}$ is locally equivalent to the initial quasimetric~$d_2$.
\end{lem}

\begin{proof}
Denote the quasimetric of type $d_2$ and the Carnot--Carath\'{e}odory metric constructed with respect to the basis $\{Y_j\}_{j=1}^{N}$ by $d_2^Y$ and $d_{cc}^Y$ respectively. The Ball-Box Theorem of {\cite{vk_birk, k_gafa}} implies that $d_2$ and $d_{cc}$ are locally equivalent. Moreover, since the lengths of a curve calculated with respect to different bases are bi-Lipschitz equivalent, the metrics $d_2^Y$ and $d_{cc}^Y$ are bi-Lipschitz equivalent as well. Considering that $d_{cc}$ and $d_{cc}^Y$ are locally equivalent, we infer that so are $d_2$ and $d_2^Y$.
\end{proof}

\begin{proof}[Proof of Theorem ${\ref{charsetth}}$]
Observe that the condition $\mathcal
H^{N}(\chi)=\mathcal H^{\nu}(\chi)=0$ and
{\cite[2.10.25]{F}} immediately imply the result.

Assume now that $\mathcal H^{N}(\chi)>0$.

{\bf Step I.} Divide $\chi$ into finitely many sets, each of which is defined by the structure of $\ker D\varphi$. In other words, represent $\chi$ as
$$
\chi=\bigcup\limits_{p_1,\ldots,p_{M}}\chi_{p_1\ldots p_{M}},
$$
where $$\chi_{p_1\ldots p_{M}}=\Bigl\{x\in\chi:\
\dim({H}_i(x)\cap\ker{D}\varphi(x))=p_i,\ i=1,\ldots, M\Bigr\}.
$$
The assumption $\mathcal H^{N}(\chi)>0$ implies the existence of a collection $\{p_1,\ldots,p_{M}\}$ with $\mathcal
H^{N}(\chi_{p_1\ldots p_{M}})>0$ since $\mathcal H^{N}(\chi)\leq\sum\limits_{p_1,\ldots,p_{M}}\mathcal
H^{N}(\chi_{p_1\ldots p_{M}})$.

{\bf Step II.} Fix a collection
$\{p_1,\ldots,p_{M}\}$ with this property, the corresponding set $\chi_{p_1\ldots p_{M}}$, and $x\in\chi_{p_1\ldots p_{M}}$ (which is not necessarily a density point). The purpose of Steps {\bf II}--{\bf V} is to prove that
\begin{itemize}
\item the set $\chi_{p_1\ldots p_{M}}$ is measurable;
\item in the coordinates of the first kind the order of tangency between a level set and its tangent plane is $o(r^M)$ at almost every~$x\in\chi_{p_1\ldots p_{M}}$ (with respect to both $\mathcal H^N$ and $\mathcal H^{\nu}$),
\end{itemize}
these facts that are independent of the basis transformations described in Lemma~{\ref{equivqm}} on sufficiently small $\mathcal U\Subset\mathbb M$.
Thus, without loss of generality we may assume that at this chosen point $x$ the differential $D\varphi(x)$ is degenerate on basis vector fields $X_{i_1}(x),\ldots, X_{i_{N-\widetilde{N}}}(x)$. Consider a compact neighborhood
$U\Subset\mathbb M$ at whose points $y$ the differential $D\varphi(y)$ does not vanish on the same vector fields as at $x$.

{\bf Step III.} On Step {\bf IV}, we are going to prove that $\mathcal
H^{N}(\chi_{p_1\ldots p_{M}})$ is measurable. To this end, we now construct a new auxiliary basis $\{Y_j\}_{j=1}^{N}$ in the same way as in Lemma {\ref{equivqm}}. This construction consists of several steps.

{\bf (i)} The purpose of this substep is to obtain a new horizontal basis the nonzero images of whose vectors under $D\varphi$ at $\chi_{p_1\ldots
p_{M}}$ are independent.

The definition of $\chi_{p_1\ldots p_{M}}$ and the structures of $D\varphi$ and $\widehat{D}\varphi$ (see {\eqref{subriemdiff}} and {\eqref{riemdiff}}) imply that
$$\dim (\ker\widehat{D}\varphi(x)\cap {H}_1(x))=\dim (\ker
D\varphi(x)\cap H_1(x))=p_1,
$$ where $p_1>n_1-\tilde{n}_1$ (see Proposition {\ref{propdiff}}, claim {\bf I (a)}). Then, the images of the remaining $n_1-p_1$ horizontal vectors are linearly independent on $\varphi(U)$.

Next, we consider the matrix of the classical differential
$D\varphi$ written with respect to the basis vector fields
$X_1,\ldots,X_{N}$ in the preimage and
$\widetilde{X}_1,\ldots,\widetilde{X}_{\widetilde{N}}$ in the image.

Consider the $\tilde{n}_1\times n_1$ block of the matrix of $D\varphi$ that corresponds to the horizontal subspaces. Identify its $n_1$ columns as vectors in $\mathbb R^{\tilde{n}_1}$. Then, at $x$ we choose from this collection $n_1-p_1$ linearly independent vectors $w_1(x),\ldots,w_{n_1-p_1}(x)$. At a point $y\in U$ we denote the corresponding column vector as $w_1(y),\ldots,w_{n_1-p_1}(y)$. The remaining $p_1$ vectors $v_1(y),\ldots,v_{p_1}(y)$ at $y\in\chi_{p_1\ldots p_{M_1}}\cap U$ are contained in the linear span of the first ones. Consequently, for the set $\chi_{p_1\ldots p_{M}}$ there exists a transformation $O^{\tilde{n}_1}_{n_1-p_1}(y)$ of $\mathbb R^{\tilde{n}_1}$ depending $C^{M+1}$-smoothly on $y\in U$ such that the images of all $w_1(y),\ldots,w_{n_1-p_1}(y)$ belong to $\mathbb R^{n_1-p_1}\times 0^{\tilde{n}_1-(n_1-p_1)}$.
To see that this mapping is $C^{M+1}$-smooth, it suffices to write the coordinates of the linearly independent vectors $\{w_1(y),\ldots,w_{n_1-p_1}(y), e_{i_1},\ldots, e_{i_{\tilde{n}_1-(n_1-p_1)}}\}$ in the standard basis as a matrix and then to calculate the inverse matrix. (Here $\{e_{i_1},\ldots, e_{i_{\tilde{n}_1-(n_1-p_1)}}\}$ are the vectors of the standard basis in $\mathbb R^{\tilde{n}_1}$ linearly independent together with $w_1(x),\ldots,w_{n_1-p_1}(x)$; without loss of generality  we may assume that the neighborhood $U$ is sufficiently small for the vectors $\{e_{i_1},\ldots, e_{i_{\tilde{n}_1-(n_1-p_1)}}\}$ to be independent together with $w_1(x),\ldots,w_{n_1-p_1}(x)$ for all $y\in U$.) The resulting matrix is the matrix of this transformation.
In other words, the matrix
$$
(O^{\tilde{n}_1}_{n_1-p_1}(y)[w_1(y)],\ldots, O^{\tilde{n}_1}_{n_1-p_1}(y)[w_{n_1-p_1}(y)])
$$
has exactly $\tilde{n}_1-(n_1-p_1)$ vanishing columns. Since we can apply to these vectors an orthogonal transformation independent of $y\in U$ without loss of generality we may assume that these vanishing columns have indices from $n_1-p_1+1$ to~$\tilde{n}_1$.

By the linear dependence, at the points of $\chi_{p_1\ldots p_{M}}$ the images of the column vectors
$v_1(y),\ldots,v_{p_1}(y)$ also belong to $\mathbb R^{n_1-p_1}$.
Represent these images (at an arbitrary point $y\in U$) as
$$
O^{\tilde{n}_1}_{n_1-p_1}(y)[v_k(y)]=(\omega_k^1(y),\omega_k^2(y))^T,
$$
where the dimension of $\omega_k^1(y)$ equals $n_1-p_1$, and the dimension of
$\omega_k^2(y)$ equals $\tilde{n}_1-(n_1-p_1)$. Since $O^{\tilde{n}_1}_{n_1-p_1}(y)\in C^M(U)$, where $U$ is a compact neighborhood of the origin, $X_i\in C^{M+1}$, $i=1,\ldots,N$, it follows that
$\omega_k^1(y)$ depends $C^M$-smoothly on $n_1-p_1$ columns. Hence, there exist smooth coefficients
$b_{k,1}(y),\ldots,b_{k,n_1-p_1}(y)$ with
$\omega_k^1(y)=\sum\limits_{l=1}^{n_1-p_1}b_{k,l}(y)w_l(y)$.
Then if the columns $w_1(y),\ldots,w_{n_1-p_1}(y)$ correspond to the vectors $X_{i_1}(y),\ldots,X_{i_{n_1-p_1}}(y)$, and a column $v_k(y)$
corresponds to $X_{j_k}(y)$, then upon replacing
$X_{j_k}(y)$ by
\begin{equation}\label{change}
X_{j_k}(y)-\sum\limits_{l=1}^{n_1-p_1}b_{k,l}(y)X_{i_l}(y)
\end{equation}
we find that the part of the new column corresponding to $\omega_k^1(y)$ equals zero, where $k=1,\ldots,p_1$.
Next, since at the points of $\chi_{p_1\ldots p_{M}}$ the columns $v_1(y),\ldots,v_{p_1}(y)$ belong to $\mathbb
R^{n_1-p_1}$ we have~$\omega_k^2(y)=0$ on $\chi_{p_1\ldots p_{M_1}}$. Apply the transformation {\eqref{change}}
for all $k=1,\ldots,p_1$. In this case, each new column
$v_k(y)$ vanishes at $y\in\chi_{p_1\ldots p_{M}}$,
$k=1,\ldots,p_1$.

Thus, we obtain a new horizontal basis the nonzero images of whose vectors under $D\varphi$ at $\chi_{p_1\ldots
p_{M}}$ are independent.

{\bf (ii)} Next, we apply similar arguments to the subspace
$H_2$ using the fact that we have already constructed a basis for $H_1$. Namely,
at $y\in\chi_{p_1\ldots p_{M}}$, in each of $p_2-p_1$ vector fields, we delete the parts depending on $n_2-p_2$ linearly independent vector fields. This yields a new basis for $H_2$ the nonzero images of whose nonhorizontal vector fields independent and nonhorizontal (belonging to $\widetilde{H}_2$).

Continuing similarly, we obtain new bases for all
$H_k$, $k=3,\ldots, M$. Denote these vector fields by $\{Y_j\}_{j=1}^{N}$.

{\bf Step IV.} By step {\bf III}, the intersection of  $\chi_{p_1\ldots p_{M}}$ with the closure of every compact neighborhood $W\Subset U$ is closed. Indeed,
take $\{y_l\}\to y$ as $l\to\infty$, where $y_l\in\chi_{p_1\ldots
p_{M}}\cap\overline{W}$. Since $D\varphi(y_l)$ at each point $y_l$ vanishes on a certain collection of basis vector fields by the definition of $\chi_{p_1\ldots p_{M}}$, the continuity of the differential implies that $D\varphi(y)$ vanishes on the same vector fields and is non-degenerate on the rest since $W\Subset U$. Thus, the set $\chi_{p_1\ldots p_M}$ is measurable.

The measurability of $\chi_{p_1\ldots p_{M}}$ implies that almost all its points are density points (both with respect to $\mathcal
H^{N}$ with balls in the Riemannian metric and
$\mathcal H^{\nu}$ with~$\operatorname{Box}_2$-balls).

{\bf Step V.} Estimate now the $\mathcal H^{N-\widetilde{N}}$-measure of the intersection of the level set passing through a density point $x\in\chi_{p_1\ldots p_{M}}$ of $\chi_{p_1\ldots
p_{M}}$, and the ball  $\operatorname{Box}_2(x,r)$. To this end, we use in the image and preimage the normal coordinates with respect to the points $\varphi(x)$ and $x$, and estimate the order of tangency between this level set and the tangent plane. We denote the resulting composition $\theta_{\varphi(x)}^{-1}\circ\varphi\circ\theta_x$ by $\psi$.

Evaluate $\psi$ at an arbitrary point
$y=\sum\limits_{i=1}^{N-\widetilde{N}}t_iZ_{j_i}(0)=\sum\limits_{i=1}^{N-\widetilde{N}}t_ie_{j_i}$,
where $Z_{j_i}=D\theta_x^{-1}[Y_{j_i}]$ are the basis vector fields tangent to the level set
$\psi^{-1}(\psi(0))$ at zero. We have {\cite{nsw}}
\begin{multline*}
\psi(y)=\sum\limits_{j=1}^{M} \Bigl[\frac{1}{j!}
\Bigl(\sum\limits_{i=1}^{N-\widetilde{N}}t_iZ_{j_i}\Bigr)^j\psi\Bigr](0)+o(|y|^{M})
\\=\sum\limits_{j=2}^{M} \Bigl[\frac{1}{j!}
\Bigl(\sum\limits_{i=1}^{N-\widetilde{N}}t_iZ_{j_i}\Bigr)^j\psi\Bigr](0)+o(|y|^{M}),
\end{multline*}
where $o(\cdot)$ is uniform on $U$. Since $0$ is a density point of zeroes of $D\psi Z_{j_i}$, the result of the action of every differential operator on it vanishes at $0$ as well:
$$\Bigl(\sum\limits_{i=1}^{N-\widetilde{N}}t_iZ_{j_i}\Bigr)^2\psi(0)=0.
$$
Similar statements hold for almost all (with respect to both $\mathcal H^N$ and $\mathcal H^{\nu}$) points of $\chi_{p_1\ldots p_{M}}$. Consequently, $0$ is a density point of the zeroes of the function $$
f(y)=\Bigl(\sum\limits_{i=1}^{N-\widetilde{N}}t_iZ_{j_i}\Bigr)^2\psi(y).
$$ Then we obtain
\begin{multline*}
\psi(y)=\sum\limits_{j=3}^{M}
\Bigl[\frac{1}{j!}\Bigl(\sum\limits_{i=1}^{N-\widetilde{N}}t_iZ_{j_i}\Bigr)^j\psi\Bigr](0)
+o(|y|^{M})\\=\sum\limits_{j=4}^{M}
\Bigl[\frac{1}{j!}\Bigl(\sum\limits_{i=1}^{N-\widetilde{N}}t_iZ_{j_i}\Bigr)^j\psi\Bigr](0)
+o(|y|^{M})
\end{multline*}
since $\Bigl(\sum\limits_{i=1}^{N-\widetilde{N}}t_iZ_{j_i}\Bigr)^3\psi(0)=0$.

Applying similar arguments, we arrive at $\psi(y)=o(|y|^{M})$.
This implies that the tangent plane approximates the level set up to $o(|y|^{M})$. Indeed, take
$y$ in the same level set as $0$, and the orthogonal projection $y^{\prime}$ of $y$ along $(\ker
D\psi(0))^{\bot}$ to the tangent plane. Then
$$
\psi(y)=\psi(y^{\prime})+o(|y|^{M})=\psi(y^{\prime})+D\psi(y^{\prime})(y-y^{\prime})+o(|y-y^{\prime}|)=\psi(y^{\prime})+C|y-y^{\prime}|,
$$ and since $|C|\geq C_0>0$ uniformly on $U$, we have
$|y-y^{\prime}|=o(|y|^{M})$. Observe that here $o(\cdot)$ is uniform on $U$. Consequently, Theorem {\ref{tang_meas}} yields
$$\mathcal H^{N-\widetilde{N}}
(\psi^{-1}(\psi(0))\cap\operatorname{Box}_2(0,r))=Lr^{\nu-\nu_0},
$$
where $0<L_1\leq L\leq L_2<\infty$ uniformly on $U$,
and
\begin{equation}\label{intersect}
\mathcal H^{N-\widetilde{N}}
(\varphi^{-1}(\varphi(x))\cap\operatorname{Box}_2(x,r))=C_xr^{\nu-\nu_0},
\end{equation}
$\nu_0>\widetilde{\nu}$. Here the number $\nu_0$ is defined by the set
$\chi_{p_1\ldots p_{M}}$,independently of $x$. Since the order of tangency is independent of the basis, the box in {\eqref{intersect}} is taken in the initial quasimetric $d_2$ which is constructed with respect to the initial basis $\{X_i\}_{i=1}^N$. Moreover, the relation
$$
\mathcal H^{N-\widetilde{N}}
(\varphi^{-1}(\varphi(y))\cap\operatorname{Box}_2(x,r))=C_yr^{\nu-\nu_0}
$$
with $0<K_1\leq C_y\leq K_2<\infty$ holds for all points
$y\in\chi_{p_1\ldots
p_{M}}\setminus\Sigma_{\chi_{p_1\ldots p_{M}}}$ for some set $\Sigma_{\chi_{p_1\ldots p_{M}}}$ with $\mathcal
H^{N}(\Sigma_{\chi_{p_1\ldots p_{M}}})=\mathcal
H^{\nu}(\Sigma_{\chi_{p_1\ldots p_{M}}})=0$.

{\bf Step VI.} Suppose that $z\in\widetilde{\mathbb M}$. Estimate the measure
$$
\mathcal H^{\nu-\widetilde{\nu}}(\varphi^{-1}(z)\cap(\chi_{p_1\ldots
p_{M}}\setminus\Sigma_{\chi_{p_1\ldots p_{M}}})\cap W)
$$
for $W\Subset U$. Since $d_{2}$ and
$d_{cc}$ are locally equivalent, the conditions of 5$r$-covering lemma are fulfilled for the balls centered at the points of $\chi_{p_1\ldots p_{M}}\setminus\Sigma_{\chi_{p_1\ldots
p_{M}}}$. Consider a covering $\bigcup\limits_{i\in\mathbb
N}\operatorname{Box}_2(x_i,r_i)$ of the set
$$\varphi^{-1}(z)\cap(\chi_{p_1\ldots
p_{M}}\setminus\Sigma_{\chi_{p_1\ldots p_{M}}})\cap W$$
such that the balls $\{\operatorname{Box}_2(x_i,r_i/5l)\}_{i\in\mathbb N}$ are disjoint. Here the number $l$ depends on the constants in the equivalence of $d_2$ and $d_{cc}$.
Then we have
\begin{multline*}
\mathcal
H^{\nu-\widetilde{\nu}}_{\delta}(\varphi^{-1}(z)\cap(\chi_{p_1\ldots
p_{M}}\setminus\Sigma_{\chi_{p_1\ldots p_{M}}})\cap W)\\
\leq C\delta^{\nu_0-\widetilde{\nu}}\sum\limits_{i\in\mathbb
N}{r_i}^{\nu-\nu_0}\leq C_1\delta^{\nu_0-\widetilde{\nu}}\sum\limits_{i\in\mathbb
N}\mathcal
H^{N-\widetilde{N}}(\varphi^{-1}(\varphi(x))\cap\operatorname{Box}_2(x_i,r_i))\\
\leq C_2\delta^{\nu_0-\widetilde{\nu}}\sum\limits_{i\in\mathbb N}\mathcal
H^{N-\widetilde{N}}(\varphi^{-1}(\varphi(x))\cap\operatorname{Box}_2(x_i,r_i/5l))
\leq\delta^{\nu_0-\widetilde{\nu}}L\to0
\end{multline*}
as $\delta\to0$ since $L\leq\mathcal
H^{N-\widetilde{N}}(\varphi^{-1}(z)\cap(\chi_{p_1\ldots
p_{M}}\setminus\Sigma_{\chi_{p_1\ldots p_{M}}})\cap W)$ is independent of $\delta>0$. Thus, the intersection of every level set with $\chi_{p_1\ldots
p_{M}}\setminus\Sigma_{\chi_{p_1\ldots p_{M}}}$ has
$\mathcal H^{\nu-\widetilde{\nu}}$-measure zero.
Consequently,
$$
\mathcal
H^{\nu-\widetilde{\nu}}(\varphi^{-1}(z)\cap(\chi\setminus\Sigma))=0
$$
where
$\Sigma=\bigcup\limits_{p_1\ldots p_{M}}\Sigma_{\chi_{p_1\ldots
p_{M}}}$ and $\mathcal H^{N}(\Sigma)=\mathcal
H^{\nu}(\Sigma)=0$. Now the theorem follows from~{\cite[2.10.25]{F}}.
\end{proof}

\section{ The Degeneration Set}

\begin{thm}\label{degen} For ${\cal H}^{\tilde{\nu}}$-almost all $t\in\widetilde{\mathbb M}$ we have
$$
{\cal
H}^{\nu-\tilde{\nu}}(\varphi^{-1}(t)\cap Z)=0.
$$
\end{thm}

\begin{proof} {\bf Step I.} We have $Z=Z_1\cup Z_2$, where ${\cal
H}^{\widetilde{N}}(\varphi(Z_1))={\cal H}^{\tilde{\nu}}(\varphi(Z_1))=0$, and
${\cal H}^{N-\widetilde{N}}(\varphi^{-1}(t)\cap Z_2)=0$ for all
$t\in\widetilde{\mathbb M}$ (see, for instance, {\cite{Km4}}).

{\bf Step II.} Note that ${\cal H}^{\widetilde{N}}(\varphi(Z_2))<\infty$ since otherwise we can represent it as a countable union of sets of finite measure. Next,
decompose $Z_2$ as $\bigcup\limits_{i=0}^{\widetilde{N}-1}C_i$, where $\operatorname{rank}
D\varphi(x)=i$ for $x\in C_i$.

Fix $0\leq i\leq \widetilde{N}-1$. Without loss of generality we may assume that $C_i$ is a compact set. Consider a compact set $\widetilde{C}_i\subset C_i$ such that $\varphi(\widetilde{C}_i)$ is a compact subset of the set of density points of $\varphi(C_i)$ and $\mathcal H^{\widetilde{N}}(\varphi(C_i)\setminus\varphi(\widetilde{C}_i))<\varepsilon$ for a~fixed arbitrary $\varepsilon>0$.
Consider the mapping
$$
\xi_i(x)=(\varphi(x),x_{N-i+1},\ldots,x_{N}):\mathbb
R^{N}\to\mathbb R^{\widetilde{N}}\times\mathbb R^{N-i}.
$$
Observe that $\xi_i({C}_i)$ is an $N$-rectifiable set.
Our assumption implies that $\xi_i$ is locally bi-Lipschitz on
${C}_i$. Consequently, ${\cal H}^{N-\widetilde{N}}(\xi_i(\varphi^{-1}(t)\cap
\widetilde{C}_i))=0$ for all $t\in\widetilde{\mathbb M}$. Consider the mapping $\pi:\xi({C}_i)\to\mathbb R^{\widetilde{N}}$ defined as
$$
\pi(x_1,\ldots, x_{\widetilde{N}+N-i})=(x_1,\ldots, x_{\widetilde{N}}).
$$
Apply the classical coarea formula to $\pi$, $C_i$ and~$\widetilde{C}_i$:
\begin{multline}\label{coareapi}
\int\limits_{\xi({C}_i)}\mathcal J(\pi,x)\,d\mathcal H^{N}(x)\geq\int\limits_{\xi({C}_i)}\chi_{\xi(\widetilde{C}_i)}(x)\mathcal J(\pi,x)\,d\mathcal H^{N}(x)\\
=\int\limits_{\varphi({C}_i)}\,d\mathcal H^{\widetilde{N}}(t)\int\limits_{\pi^{-1}(t)\cap\xi({C}_i)}\chi_{\xi(\widetilde{C}_i)}(u)\,d\mathcal H^{N-\widetilde{N}}(u).
\end{multline}
By the definition of $\widetilde{C}_i$, we find that $\mathcal J(\pi,x)=1$ for $x\in\xi(\widetilde{C}_i)$. Furthermore,
$$\mathcal H^{N-\widetilde{N}}(\pi^{-1}(t)\cap\xi({C}_i))={\cal H}^{N-\widetilde{N}}(\xi_i(\varphi^{-1}(t)\cap
\widetilde{C}_i))=0.
$$ It follows from {\eqref{coareapi}} that $\mathcal H^N(\xi(\widetilde{C}_i))=0$. Finally,
have
$$
{\cal H}^{N}(\widetilde{C}_i)={\cal H}^{\nu}(\widetilde{C}_i)=0
$$
because $\xi_i$ is bi-Lipschitz.
Since $\varepsilon>0$ is arbitrary, there exists a set $\Sigma_i$ of measure zero in $\varphi(C_i)$ such that
$$\mathcal H^N(\varphi^{-1}(\varphi(C_i)\setminus\Sigma_i)\cap C_i)=\mathcal H^{\nu}(\varphi^{-1}(\varphi(C_i)\setminus\Sigma_i)\cap C_i)=0.
$$

Therefore, ${\cal
H}^{\nu}\Bigl(\bigcup\limits_{i=0}^{\widetilde{N}-1}\varphi^{-1}(\varphi(C_i)\setminus\Sigma_i)\cap C_i\Bigr)=0$.

{\bf Step III.} Since $\varphi$ is contact, it is locally
Lipschitz with respect to every sub-Riemannian metric. Reasoning as in {\cite[Theorem 2.10.25]{F}}, we obtain
\begin{multline*}
0\leq\int\limits_{\widetilde{\mathbb M}}\,d{\cal
H}^{\tilde{\nu}}(t)\int\limits_{\varphi^{-1}(t)\cap \bigcup\limits_{i=0}^{\widetilde{N}-1}\varphi^{-1}(\varphi(C_i)\setminus\Sigma_i)\cap C_i}\,d{\cal
H}^{\nu-\tilde{\nu}}(u)\\
\leq C(\operatorname{Lip}(\varphi),\nu,\tilde{\nu}){\cal
H}^{\nu}\Bigl(\bigcup\limits_{i=0}^{\widetilde{N}-1}\varphi^{-1}(\varphi(C_i)\setminus\Sigma_i)\cap C_i\Bigr)=0.
\end{multline*}
Furthermore, the $\mathcal H^{\widetilde{N}}$-measure zero set $\bigcup\limits_{i=0}^{\widetilde{N}-1}\Sigma_i$ is included into an $\mathcal H^{\widetilde{N}}$-measure zero Borel set $\Sigma$. Then, its preimage $\varphi^{-1}(\Sigma)$ is also a Borel set, and the set $\varphi^{-1}(\Sigma)\cap Z_2$ is measurable. By the definition,
$$\mathcal H^{\widetilde{N}}(\varphi(\varphi^{-1}(\Sigma)\cap Z_2))=\mathcal H^{\widetilde{\nu}}(\varphi(\varphi^{-1}(\Sigma)\cap Z_2))=0.
$$ combining the results, we deduce that
$$
\int\limits_{\widetilde{\mathbb M}}\,d{\cal
H}^{\tilde{\nu}}(t)\int\limits_{\varphi^{-1}(t)\cap Z}\,d{\cal
H}^{\nu-\tilde{\nu}}(u)=0,
$$
and thus we complete the proof.
\end{proof}

\section{The Coarea Formula}

\begin{thm}[Coarea Formula] If a contact mapping
$\varphi:\mathbb M\to\widetilde{\mathbb M}$ satisfies Assumption~${\ref{assumpphi}}$ then the coarea formula holds:
\begin{equation}\label{finalcoarea}
\int\limits_{\mathbb M}{\cal J}_{\widetilde{N}}^{SR}(\varphi,x)\,d{\cal
H}^{\nu}(x) =\int\limits_{\widetilde{\mathbb M}}\,d{\cal
H}^{\tilde{\nu}}(t)\int\limits_{\varphi^{-1}(t)}\,d{\cal
H}^{\nu-\tilde{\nu}}(u).
\end{equation}
\end{thm}

\begin{proof}
Obvserve that $\mathbb M=\mathbb D\sqcup\chi\sqcup Z$. Consequently,
\begin{multline*}
\int\limits_{\mathbb M}{\cal J}_{\widetilde{N}}^{SR}(\varphi,x)\,d{\cal
H}^{\nu}(x) =\int\limits_{\mathbb D}{\cal J}_{\widetilde{N}}^{SR}(\varphi,x)\,d{\cal
H}^{\nu}(x)\\
+\int\limits_{\chi}{\cal J}_{\widetilde{N}}^{SR}(\varphi,x)\,d{\cal
H}^{\nu}(x) +\int\limits_{Z}{\cal J}_{\widetilde{N}}^{SR}(\varphi,x)\,d{\cal
H}^{\nu}(x).
\end{multline*}
Theorem {\ref{coaread}} yields
$$
\int\limits_{\mathbb D}{\cal
J}^{SR}_{\widetilde{N}}(\varphi,x)\,d{\cal
H}^{\nu}(x)=\int\limits_{\widetilde{\mathbb M}}\,d{\cal
H}^{\tilde{\nu}}(t)\int\limits_{\varphi^{-1}(t)\cap\mathbb
D}\,d{\cal H}^{\nu-\tilde{\nu}}(u).
$$
Theorem {\ref{charsetth}}, implies
$$
\int\limits_{\chi}{\cal J}_{\widetilde{N}}^{SR}(\varphi,x)\,d{\cal
H}^{\nu}(x)=0=\int\limits_{\widetilde{\mathbb M}}\,d{\cal
H}^{\tilde{\nu}}(t)\int\limits_{\varphi^{-1}(t)\cap\chi}\,d{\cal
H}^{\nu-\tilde{\nu}}(u).
$$
Finally, since ${\cal J}_{\widetilde{N}}^{SR}(\varphi,x)=0$ on $Z$, Theorem {\ref{degen}} yields
$$
\int\limits_{Z}{\cal J}_{\widetilde{N}}^{SR}(\varphi,x)\,d{\cal
H}^{\nu}(x)=0=\int\limits_{\widetilde{\mathbb M}}\,d{\cal
H}^{\tilde{\nu}}(t)\int\limits_{\varphi^{-1}(t)\cap Z}\,d{\cal
H}^{\nu-\tilde{\nu}}(u).
$$
Combinig these relations, we obtain {\eqref{finalcoarea}}. The theorem follows.
\end{proof}

\end{document}